\documentclass[reqno,10pt]{amsart}

\usepackage[a4paper,left=27mm,right=27mm,top=27mm,bottom=27mm,marginpar=25mm]{geometry} 
\usepackage{amsmath}
\usepackage{amssymb}
\usepackage{amsthm}
\usepackage{amscd}
\usepackage[ansinew]{inputenc}
\usepackage{color}
\usepackage[final]{graphicx}
\usepackage{tikz}
\usepackage{mathabx}

\renewcommand{\epsilon}{\varepsilon}
\renewcommand{\phi}{\varphi}

\numberwithin{equation}{section}

\newtheoremstyle{thmlemcorr}{10pt}{10pt}{\itshape}{}{\bfseries}{.}{10pt}{{\thmname{#1}\thmnumber{ #2}\thmnote{ (#3)}}}
\newtheoremstyle{thmlemcorr*}{10pt}{10pt}{\itshape}{}{\bfseries}{.}\newline{{\thmname{#1}\thmnumber{ #2}\thmnote{ (#3)}}}
\newtheoremstyle{defi}{10pt}{10pt}{\itshape}{}{\bfseries}{.}{10pt}{{\thmname{#1}\thmnumber{ #2}\thmnote{ (#3)}}}
\newtheoremstyle{remexample}{10pt}{10pt}{}{}{\bfseries}{.}{10pt}{{\thmname{#1}\thmnumber{ #2}\thmnote{ (#3)}}}
\newtheoremstyle{ass}{10pt}{10pt}{}{}{\bfseries}{.}{10pt}{{\thmname{#1}\thmnumber{ A#2}\thmnote{ (#3)}}}

\theoremstyle{thmlemcorr}
\newtheorem{theorem}{Theorem}
\numberwithin{theorem}{section}
\newtheorem{lemma}[theorem]{Lemma}

\newtheorem{proposition}[theorem]{Proposition}

\theoremstyle{thmlemcorr*}
\newtheorem{theorem*}{Theorem}
\newtheorem{lemma*}[theorem]{Lemma}
\newtheorem{corollary*}[theorem]{Corollary}
\newtheorem{proposition*}[theorem]{Proposition}
\newtheorem{problem*}[theorem]{Problem}
\newtheorem{conjecture*}[theorem]{Conjecture}

\theoremstyle{defi}
\newtheorem{definition}[theorem]{Definition}

\theoremstyle{remexample}
\newtheorem{remark}[theorem]{Remark}
\newtheorem{example}[theorem]{Example}

\theoremstyle{ass}

\newcommand{\Acal}{\mathcal{A}}
\newcommand{\cA}{\mathcal{A}}
\newcommand{\Bcal}{\mathcal{B}}

\newcommand{\Hcal}{\mathcal{H}}

\newcommand{\Pcal}{\mathcal{P}}

\newcommand{\Qcal}{\mathcal{Q}}

\newcommand{\Scal}{\mathcal{S}}

\newcommand{\Ucal}{\mathcal{U}}
\newcommand{\Vcal}{\mathcal{V}}

\newcommand{\Abb}{\mathbb{A}}

\newcommand{\NN}{\mathbb{N}}

\newcommand{\Pbb}{\mathbb{P}}

\newcommand{\RR}{\mathbb{R}}

\newcommand{\Tbb}{\mathbb{T}}
\newcommand{\TT}{\mathbb{T}}
\newcommand{\T}{\mathbb{T}}

\newcommand{\Zbb}{\mathbb{Z}}
\newcommand{\ZZ}{\mathbb{Z}}

\DeclareMathOperator{\diverg}{div}
\DeclareMathOperator{\curl}{curl}

\DeclareMathOperator{\rank}{rank}

\DeclareMathOperator{\supp}{supp}

\newcommand{\norm}[1]{\|#1\|}

\newcommand{\normb}[1]{\bigl\|#1\bigr\|}
\newcommand{\normB}[1]{\Bigl\|#1\Bigr\|}

\newcommand{\abs}[1]{\left|#1\right|}

\newcommand{\absb}[1]{\bigl|#1\bigr|}
\newcommand{\absB}[1]{\Bigl|#1\Bigr|}

\newcommand{\dd}{\;\mathrm{d}}

\newcommand{\N}{\mathbb{N}}
\newcommand{\R}{\mathbb{R}}

\newcommand{\weakly}{\rightharpoonup}

\newcommand{\To}{\longrightarrow}
\newcommand{\todown}{\downarrow}

\newcommand{\eps}{\epsilon}

\newcommand{\sbullet}{\begin{picture}(1,1)(-0.5,-2)\circle*{2}\end{picture}}
\newcommand{\frarg}{\,\sbullet\,}

\newcommand{\fhom}{f_{\rm hom}^\Acal}
\newcommand{\fhomzero}{f_{\rm hom}^{\Acal_0}}
\newcommand{\fhomprime}{f_{\rm hom}^{\Acal'}}
\newcommand{\alphaweakly}{\stackrel{r3\text{-}s^\alpha}{\rightharpoonup} }

\def\Xint#1{\mathchoice 
{\XXint\displaystyle\textstyle{#1}}% 
{\XXint\textstyle\scriptstyle{#1}}% 
{\XXint\scriptstyle\scriptscriptstyle{#1}}% 
{\XXint\scriptscriptstyle\scriptscriptstyle{#1}}% 
\!\int} 
\def\XXint#1#2#3{{\setbox0=\hbox{$#1{#2#3}{\int}$} 
\vcenter{\hbox{$#2#3$}}\kern-.5\wd0}} 
\def\dashint{\,\Xint-}

\title[Heterogeneous thin films]{Heterogeneous thin films: Combining homogenization and dimension reduction with directors}

\author{Carolin Kreisbeck}
\address{Fakult\"at f\"ur Mathematik, Universit\"at Regensburg, 93040 Regensburg, Germany}
\email{Carolin.Kreisbeck@mathematik.uni-regensburg.de}

\author{Stefan Kr\"omer}
\address{Mathematisches Institut, Universit\"at zu K\"oln, 50923 K\"oln, Germany}
\email{skroemer@math.uni-koeln.de}
\begin{document}

%%%%%%%%%%%%%%%%%%%%%% ABSTRACT %%%%%%%%%%%%%%%%%%%%%%%%%%%%%%%%%

 \begin{abstract}
 \vspace{8pt}                               
We analyze the asymptotic behavior of a multiscale problem given by a sequence of integral functionals subject to differential constraints conveyed by a constant-rank operator with two characteristic length scales, namely the film thickness and the period of oscillating microstructures, by means of $\Gamma$-convergence. 
On a technical level, this requires a subtle merging of homogenization tools, such as multiscale convergence methods, with dimension reduction techniques for functionals subject to differential constraints.
One observes that the results depend critically on the relative magnitude between the two scales. Interestingly, this even regards the
fundamental question of locality of the limit model, and, in particular, leads to new findings also in the gradient case.
 \vspace{8pt}

 
   \vspace{8pt}\end{abstract}
 
\maketitle

%%%%%%%%%%%%%%%%%%%%%%%%%%%%%%%%%%%%%%%%%%%%%%%%%%% Introduction %%%%%%%%%%%%%%%%%%%%%%%%%%%%%%%%%%%%%%%%%%%%%%%%%%%%%%

\section{Introduction}
Given two characteristic lengths
$\eps>0$ and $\eps^\alpha$ with some fixed power $\alpha>0$, 
we study the asymptotic behavior of functionals of the form
\begin{align}\label{FGradthin} 
	v\mapsto \frac{1}{\eps}\int_{\Omega_\eps} f\Bigl(\frac{x}{\eps^{\alpha}}, V(x)\Bigr)\dd{x},\quad \text{where $V=\nabla v$ with $v:\Omega_\eps\to \R^n$, }
\end{align}
as $\eps \to 0$.

Integrals of this type are standard models for the internal elastic energy of a thin film of hyperelastic, heterogeneous material. In this framework,
$\Omega_\eps:=\omega\times (0,\eps)\subset \R^d$ is the reference configuration of the film, which is thin in one direction (with thickness $\eps$), and the 
explicit dependence of $f$ on $\frac{x}{\eps^{\alpha}}$, which is assumed to be periodic in this variable, represents a material inhomogeneity of length scale $\eps^{\alpha}$. 
The scaling factor preceding the integral is related to the strength of the applied external forces added to the model. 
Rescaling by $\frac{1}{\eps}$, as above, represents the so-called membrane regime, which is suitable for rather strong forces, typically of the order required to stretch the material.

A change of variables allows us to work in a fixed domain $\Omega_1$ instead of $\Omega_\eps$. After the parameter transformation, the energy is given by the functional $F_{\eps,\eps^\alpha}:L^p(\Omega_1;\R^{n\times d})\to [0,\infty]$,
\begin{align}\label{FGrad} 
F_{\eps,\eps^\alpha}(U)= \begin{cases}
\displaystyle\int_{\Omega_1} f\Bigl(\frac{x'}{\eps^{\alpha}},\frac{x_d}{\eps^{\alpha-1}}, U(x)\Bigr)\dd{x} & 
\text{ if $U=\bigl(\nabla' u,\frac{1}{\eps} \partial_d u\bigr)$ for a $u\in W^{1,p}(\Omega_1;\RR^n)$,}\\
\infty & \text{otherwise,}
\end{cases}
\end{align}
where $1<p<\infty$, $x=(x',x_d)\in \omega\times (0,1)=\Omega_1$, and $\nabla'u$ denotes the gradient with respect to $x'$, i.e., the first $d-1$ columns of $\nabla u$.
We use the concept of $\Gamma$-convergence to rigorously derive the limit functional of $F_{\eps, \eps^\alpha}$ as $\eps\to 0$ (see e.g.~\cite{Bra02,DalMa93B}). This limit depends heavily on the notion of convergence used for a sequence of admissible finite-scale states $(U_\eps)$ (or $(u_\eps)$ at the level of potentials), i.e.~$F_{\eps, \eps^\alpha}(U_\eps)<\infty$ for $\eps>0$, giving rise to a limit state $U_0$. The particular choice of convergence determines the nature of admissible limit states and the amount of information they carry. 
Most of the earlier results 
concerning thin-film limits of homogeneous films~\cite{BB07, BFF00} or simultaneous homogenization and dimension reduction~\cite{Shu00, BaBa06} (see also the references therein for a more extensive history) were obtained for weak convergence $u_\eps \rightharpoonup u_0$ in $W^{1,p}(\Omega_1;\R^n)$ (where $U_0=\nabla u_0$), or, essentially equivalent,
for strong convergence $u_\eps \to u_0$ in $L^p(\Omega_1;\R^n)$. Due to the factor $\frac{1}{\eps}$ in front of $\partial_d u$ in~\eqref{FGrad}, which penalizes changes in direction $x_d$, all limit potentials $u_0$ are constant in $x_d$. If $f$ is coercive in a suitable sense, compactness for sequences of finite-scale states with uniformly bounded energy is guaranteed even with respect to a stronger notion of convergence, 
namely,
\begin{align}\label{gradconv}
	U_\eps=\Big(\nabla' u_\eps\Big|\frac{1}{\eps} \partial_d u_\eps\Big)\rightharpoonup U_0=(U_0'|U_{0, d})\quad\text{weakly in $L^p(\Omega_1;\RR^{n\times d})$}.
\end{align}
It is not difficult to see that obtaining finite energy in the limit requires that $U_0'=\nabla'u_0$, i.e., the first $d-1$ columns $U_0'$ of $U_0:\Omega_1\to \RR^{n\times d}$ have to be a gradient of some function $u_0\in W^{1,p}(\Omega_1;\RR^n)$ which is constant in $x_d$. In addition, we retain the information contained in the weak limit $b:=U_{0, d}$ of the $d$th column of $U_0$, commonly called bending moment, director or Cosserat vector. 
Given $u_0$ and $b$, an associated finite-scale approximating sequence can be obtained by setting
\[
	u_\eps(x',x_d):=u_0(x')+\eps \int_0^{x_d} b(x',s)\,ds,\qquad x\in \Omega_1,
\]
even though this choice will usually not give the optimal internal energy, because it disregards possible microstructure favored due to material inhomogeneities.

In this article, we focus on a quite surprising and mathematically challenging effect linked to the Cosserat vector: the possible nonlocal character of the limit functional. 
As we will see, this is essentially rooted in the fact that, unlike the other parts of the limit state, the Cosserat vector can depend on the ``thin'' variable $x_d$.
In the context of pure dimension reduction, i.e., for the thin-film limit of homogeneous material, the appearance of this effect was first conjectured in \cite{BFM09} by Bouchitt\'{e}, Fonseca and Mascarenhas (for other settings where nonlocal limit functionals are observed we refer to~\cite{FoFraLe07,DalMaFoLe10}). In~\cite{BFM09}, nonlocal effects are related to a lack of convexity\footnote{more precisely, cross-quasiconvexity, a weaker variant of convexity} of the energy density, that is, if they appear at all, which is still unknown. In our framework, we are able to prove that the question of locality versus nonlocality of the $\Gamma$-limit depends heavily on the interplay between the two microscales in $F_{\eps,\eps^\alpha}$, even for convex $f$: If the material inhomogeneity is finer than the film thickness ($\eps^\alpha<<\eps$, i.e., $\alpha>1$), then the $\Gamma$-limit with respect to \eqref{gradconv} is always a local integral functional, whereas in the other case ($\alpha\leq 1$), there are convex energy densities $f$ for which the $\Gamma$-limit is nonlocal. It is striking that this phenomenon can even appear in the mathematically ``simple'' convex case. On the other hand, we do exploit convexity in our proof in all cases, which implies that our results do not apply to hyperelastic energy densities that prevent local interpenetration of matter.

The notion of locality used here is essentially the one of \cite{BFM09}, made precise as follows:
\begin{definition}\label{def:local}
We call a functional $F:L^p(\Omega_1;\R^m)\to [0, \infty]$ local, if it can be expressed as an integral with respect to the Lebesgue measure, i.e.,~if there exists a density $g:\Omega_1\times \R^m\to [0,\infty]$ such that  for every $U\in L^p(\Omega_1;\R^m)$ with $F(U)<\infty$, the map $x\mapsto g(x,U(x))$ is measurable
and 
\begin{align}\label{eq:local}
F(U)=\int_{\Omega_1} g(x,U(x))\dd{x}.
\end{align}
If $F$ is not local, we refer to it as a nonlocal functional. 
\end{definition}
The main feature that links Definition~\ref{def:local} to an intuitive concept of locality is the additivity of such functionals considered as a function of their domain, i.e., $D\mapsto F(U;D)$ with $D\subset \Omega_1$. We set $F(U)=F(U;\Omega_1)$. If $F$ is an integral functional as in~\eqref{eq:local} and $U\in L^p(\Omega_1;\R^m)$ is admissible in the sense that $F(U;\Omega_1)<\infty$,
then clearly 
\begin{align*}
	F(U;\Omega_1)=F(U;D)+F(U;\Omega_1\setminus D)
\end{align*}
for any open $D\subset \Omega_1$. 

For our thin-film model, we are looking at the $\Gamma$-$\liminf$ as $\eps\to 0$ as a function of the domain, that is, 
\begin{equation}\label{Gammaliminf}
	F_0^-(U_0;D)=\inf\{ \liminf_{\eps\to 0} F_{\eps,\eps^\alpha}(U_\eps;D): (U_\eps)\subset L^p(\Omega_1;\R^{n\times d}), U_\eps\weakly U_0 \text{ in $L^p(D;\R^{n \times d})$})\}
\end{equation}
where
\[
F_{\eps,\eps^\alpha}(U;D):=
\begin{cases}
\displaystyle\int_{D} f\Bigl(\frac{x'}{\eps^{\alpha}},\frac{x_d}{\eps^{\alpha-1}}, U(x)\Bigr)\dd{x} & 
\text{ if $U=\bigl(\nabla' u,\frac{1}{\eps} \partial_d u\bigr)$ for a $u\in W^{1,p}(\Omega_1;\RR^n)$,}\\
\infty & \text{otherwise.}
\end{cases}
\]
In view of the heterogeneous character of the material, the optimal energy of an ``effective'' macroscopic state $U_0$ in the limit as $\eps\to 0$ is 
typically achieved along a sequence of finite-scale states $(U_\eps)$ that develop a suitable microstructure on the length scale $\eps^\alpha$ on top of $U_0$. The additivity of the limit functional~\eqref{Gammaliminf} then means that among the various different microstructures that are locally optimal in two (or more) pieces of the film, for instance $D$ and $\Omega_1\setminus \bar{D}$, we can always choose a pair that can be combined without additional energetic cost for a transition layer,
at least as long $\partial D$ has measure zero. If, on the other hand, additivity fails, this means that locally optimal microstructures for finite $\eps$ may be incompatible with the required gradient structure at $\Omega_1\cap \partial D$. As a consequence the optimal microstructure on the full domain $\Omega_1$ may take a different form, most likely involving a large transition layer with non-negligible energetic cost that is not completely determined by the local properties of $U_0$ and $f$ alone. 

\begin{remark}
a) Definition~\ref{def:local} could be generalized by also allowing $g$ to depend on derivatives of $U$ or other derived local quantities. Indeed, this can easily be encoded by replacing the admissible states $U$ with, say, $\tilde{U}=(U,\nabla U)$, by setting $F(\tilde{U})=\infty$ on fields lacking this structure.

b) From the point of view outlined above, it might seem more natural to define ``local'' in a slightly more general way, by allowing for a different measure in the integral representation of $F$ in Definition~\ref{def:local}, not necessarily absolutely continuous with respect to the Lebesgue measure. However, the limit functional $F_0^-$ in our setting cannot have this form. This is essentially a consequence of the decomposition lemma for dimension reduction (see~\cite[Theorem~1.1]{BoFo02} or \cite[Theorem~3.1]{BZ07}), which proves that any optimal or nearly optimal sequence of states $(U_\eps)$ with  $U_\eps=(\nabla' u_\eps,\frac{1}{\eps} \partial_d u_\eps)$ in the definition of $F_0^-$ that charges a set of Lebesgue measure zero 
can always be replaced by an equiintegrable sequence with the same or less energy in the limit. 

c) In \cite[Definition~15.21]{DalMa93B}, 
$F_0^-$ is called local if for every admissible $D\subset \Omega_1$,
\begin{align}\label{local2}
	F_0^-(V;D)=F_0^-(U;D) \quad\text{for all $U,V$ such that $U=V$ a.e.~on $D$},
\end{align}
which has nothing to do with Definition~\ref{def:local} or additivity as a set function. Actually, \eqref{local2} is trivially true for
$F_0^-$ as defined in \eqref{Gammaliminf}: Any nearly optimal sequence $U_\eps=(\nabla' u_\eps|\frac{1}{\eps}\partial_d u_\eps)\rightharpoonup U$ in $L^p(D; \R^{n\times d})$ such that $\liminf_{\eps\to 0} F_{\eps,\eps^\alpha}(U_\eps;D)=F_0^-(U;D)$ up to an arbitrarily small error
is also admissible in the infimum in the definition of $F_0^-(V;D)$,
because $U_\eps\rightharpoonup U=V$ on $D$, whence $F_0^-(U;D)\geq F_0^-(V;D)$. The converse inequality is analogous. 
In our opinion this only shows that \eqref{local2} does not have a deeper meaning for limit functionals like $F_0^-$.
\end{remark}

Now assume that $\omega\subset \R^{d-1}$ is a simply connected, bounded Lipschitz domain.
For functionals of the form \eqref{FGrad}, our two main results, Theorem~\ref{theo:simultaneous_local} and Theorem~\ref{theo:simultaneous_nonlocal},
then read as follows (with the hypotheses (H0)--(H5) on $f$ introduced in Section~\ref{sec:pre}):

\begin{theorem}[Local limit for \boldmath{$\alpha>1$}]
Let $\alpha>1$ and assume that $f:\R^{d}\times \R^{n\times d}(\cong\R^m)\to \R$ satisfies (H0)--(H5).
Then, as $\eps\todown 0$, the sequence of functionals $(F_{\eps, \eps^\alpha})$ as defined in \eqref{FGrad}, 
$\Gamma$-converges with respect to weak convergence in $L^p$ as in 
\eqref{gradconv} to the limit functional $F_0: L^p(\Omega_1;\R^{n\times d})\to [0,\infty]$ defined by
\begin{align*}
F_0(U)= \begin{cases}
\displaystyle\int_{\Omega_1} f_{\rm hom}\bigl(U(x)\bigr)\dd{x} & 
\text{if $U'=\nabla' v$ for $v\in W^{1,p}(\omega;\RR^n)$,}\\
\infty & \text{otherwise,}
\end{cases}
\end{align*} 
for $U\in L^p(\Omega_1;\R^{n\times d})$. The homogenized energy density $f_{\rm hom}:\R^{n\times d}\to[0,\infty)$ is given by
\begin{equation*}
\begin{aligned}
f_{\rm hom}(\xi) &=\inf_{v\in W^{1,p}(\Tbb^d;\RR^n)} \dashint_{(0,1)^d} f\bigl(y, \xi + \nabla v(y)\bigr)\dd{y}, \qquad \xi\in \R^{n\times d}.
\end{aligned}
\end{equation*}
\end{theorem}
\begin{theorem}[Nonlocal example for \boldmath{$\alpha\leq 1$}]\label{theo:nonlocalGrad}
Let $\alpha\leq 1$, $d\geq 2$, $\omega=(0,1)^{d-1}$, and $\eps_j\todown 0$ be the sequence with $\eps_j^\alpha=\frac{1}{j}$. 
Then there exists $f:\R^{d}\times \R^{n\times d} \to \R$ satisfying (H0)-(H5) such that $F_0^-$ as defined in \eqref{Gammaliminf} is nonlocal.
\end{theorem}
\begin{remark}
The proof of Theorem~\ref{theo:simultaneous_nonlocal}
in Section~\ref{sec:alpha<1}, which provides a more general version of Theorem~\ref{theo:nonlocalGrad}, translated to the setting of \eqref{FGrad}, shows the following (see also Example~\ref{ex:nonlocaldivcurl}~b)):
For an appropriate choice of $f$, there exists $U_0=(U'_0|b)\in L^p(\Omega_1;\RR^{n\times d})$ such that $U'_0$ is constant on $\Omega_1$, $b$ is piecewise constant with a single jump across the surface $x_d=\frac{1}{2}$ and
\[
	F_0^-(U_0;D_1)=F_0^-(U_0;D_2)=0<F_0^-(U_0;\Omega_1)<\infty,
\]
where $D_1=\omega\times (0,\frac{1}{2})$, $D_2=\omega\times (\frac{1}{2},1)$.
\end{remark}
So far, we have discussed functionals whose admissible states are given as a gradient (or rescaled gradient) of some potential. For the rest of this article, we will work in a more general framework of linear differential constraints, referred to as the $\Acal$-free framework, made precise in Section~\ref{sec:pre}. This means that in \eqref{FGradthin}, the constraint $V=\nabla v$ is replaced by $\cA V=0$ in $\Omega_\eps$, where $\cA$ is a vectorial first-order differential operator with constant coefficients. By choosing the operator $\cA$ as the curl in $\RR^d$ applied row by row to functions with
values in $\RR^{n\times d}$, the constraint of \eqref{FGradthin} fits exactly into this abstract setting. Indeed,
\begin{align*} 
	V=\nabla v \quad\text{ for some $v\in W^{1,p}(\Omega_\eps;\R^n)$,}
\end{align*} 
if and only if
\begin{align*}
 (\cA V)_{ij}:=(\curl V)_{ij}=\partial_i V_j-\partial_j V_i=0\in \RR^n  \text{ on $\Omega_\eps$, for $1\leq i<j\leq d$,}
\end{align*}
where $V_j$ denotes the $j$th column of $V:\Omega^1\to  \RR^{n\times d}\cong \R^m$. 
On the rescaled domain $\Omega_1$ (compare~\eqref{FGrad}), the constraint $U=(U'|U_d)=(\nabla' u|\frac{1}{\eps} \partial_d u)$ becomes $\cA_\eps U=\curl_\eps U=0$ on $\Omega_1$, where
\[
	(\curl_\eps U)_{ij} := \left\{\begin{alignedat}{2}
	&\left(\partial_i V_j-\partial_j V_i \right)\quad&&\text{for $1\leq i<j<d$, }\\
	&\left(\partial_i V_d-\tfrac{1}{\eps}\partial_d V_i \right)\quad&&\text{for $1\leq i<j=d$.}
	\end{alignedat}\right.
\]
An overview of our main results in the general framework can be found in Section~\ref{sec:regimes&results}. The corresponding proofs are given in the remaining sections.

%%%%%%%%%%%%%%%%%%%%%%%%%%%%%%%%%%%%%%%%%%%%%%%%%%%%%% Preliminaries %%%%%%%%%%%%%%%%%%%%%%%%%%%%%%%%%%%%%%%%%%%%%%%%%%%%
\section{Preliminaries}\label{sec:pre}

\subsection{Notation}
The partial derivative of a function $u=u(x)$ with respect to the $k$th component is denoted by $\partial_k u$.  For clarity, we sometimes add the respective variable as a superscript, like $\partial_k^x u$, or $\Acal^x u$ for a differential operator $\Acal$.
By $\T^d$ we denote the $d$-torus, which results from gluing opposite edges of the unit 
square $Q^d:=(0,1)^d$. More generally, one can define for an arbitrary cuboid $Q\subset \R^d$ the corresponding torus $\T^d(Q)$. The vectors $e_1, \ldots, e_d$ constitute the standard unit basis in $\R^d$, and we use the notation $y=(y', y_d)\in \R^d$, where $y'=(y_1, y_2, \ldots, y_{d-1})$.
For $\alpha\in \R$, the number $\lceil \alpha \rceil := \min\{n\in \Zbb: \alpha<n\}$ is the smallest integer following $\alpha\in \R$, and $p':=\frac{p}{p-1}$ denotes the dual exponent of $1<p<\infty$.  
When speaking of sequences with index $\eps>0$, we mean that $\eps$ can stand for any sequence $\eps_j\todown 0$ as $j\to \infty$. Throughout the paper, constants can change from line to line.

\subsection{Problem formulation in the $\Acal$-free framework}\label{sec:formulation_Afree} 
The motivation for working in a quite general mathematical setting which, in particular, covers the problem highlighted in the introduction, is twofold.
On the one hand, we intend to cover different fields of applications for thin films with grain structure or layers such as elasticity, micromagnetics or magnetostriction. The $\Acal$-free framework is a way to treat various variational problems in mechanics and electromagnetism, as well as combinations thereof, in a unified way. 
On the other hand, the proof of the nonlocal behavior in the example given in Section~\ref{sec:alpha<1} naturally leads to this setting even in the gradient case outlined in the introduction: In that case, we actually show that certain vector fields are far away from the subspace of gradient fields by measuring their curl in the norm of $W^{-1,p}$, an argument that cannot be replicated by just using potentials.

The general approach used in this paper operates on states that satisfy a PDE constraint conveyed by a first order differential operator $\Acal$, and is rooted in the theory of compensated compactness developed by Murat and Tartar~\cite{Mur78, Mur81, Tar79}. Building on work by Dacorogna~\cite{Dac82}, Fonseca and M\"uller \cite{FM99} established the theory for variational principles where the $\Acal$-free vector fields are exactly the admissible states. Important examples include integral functionals defined on deformation gradients (using $\cA=\curl$ on a simply connected domain) 
as they appear in hyperelasticity theory, as well as variational problems on solenoidal vector fields ($\cA=\diverg$) or on solutions of the Maxwell equations. 
In the following, we merge the two processes of homogenization and dimension reduction in the $\Acal$-free framework to rigorously derive 
effective and reduced limit models for heterogeneous thin films.

Let $\Omega_\eps=\omega\times (0,\eps)\subset \RR^d$ with space dimension $d>1$ and $\omega\subset Q^{d-1}$ a bounded
Lipschitz domain be the reference configuration of a thin film with thickness $\eps>0$.
The parameter $\delta>0$ stands for the length scale of the material heterogeneity in form of a periodic grain structure.

We study the asymptotics of the constrained minimization problem
\begin{align}\label{vp_beforerescaling}
\frac{1}{\eps}\int_{\Omega_\eps} f\Bigl(\frac{y}{\delta}, v(y)\Bigr)\dd{y} \to \min,\qquad  \text{$v:\Omega_\eps\to \R^m$ with $\Acal v=0$ in $\Omega_\eps$,}
\end{align}
for various regimes as $\eps$ and $\delta$ tend to zero.

Here, $f:\R^{d}\times \R^m\to [0,\infty)$ and $\Acal$ is a linear first-order constant-coefficient partial 
differential operator with symbol
\begin{align*}
\Abb(\eta)= \sum_{k=1}^d A^{(k)}\eta_k\qquad \text{for $\eta\in \R^d$ with matrices $A^{(1)}, \ldots, A^{(d)}\in \R^{l\times m}$.}
\end{align*}

For given $1<p<\infty$ the hypotheses on $f$ are listed below:
\begin{itemize}
 \item[(H0)] {\it Regularity}\\
             $f:\R^{d}\times\R^m\to [0,\infty)$ is Caratheodory, i.e.\ $f(\frarg, \xi)$ is measurable for every 
             $\xi\in \R^m$ and $f(z,\frarg)$ is continuous for almost every $z\in \R^{d}$;
  \item[(H1)] {\it Higher regularity}\\
 $ \partial_\xi f(z, \xi)$ exists for almost all $z\in \R^{d}$ and all $\xi\in \R^m$, and satisfies (H0);
 \item[(H2)] {\it Periodicity}\\
             $f(\frarg, \xi)$ is $Q^{d}$-periodic for every $\xi \in \R^m$;
 \item[(H3)] {\it Growth }\\
             $0 \leq f(z,\xi)\leq c_1(1+\abs{\xi}^p)$ for every $(z,\xi)\in \R^{d}\times\R^m$ and a constant $c_1>0$;
 \item[(H4)] {\it Coercivity}\\
    $f(z,\xi)\geq  c_2\abs{\xi}^p-c_3$ for every $(z,\xi)\in \R^{d}\times\R^m$ with constants $c_2> 0$ and $c_3\in \R$;
 \item[(H5)]{\it Convexity}\\
  $f(z, \frarg)$ is convex for almost every $z\in \R^{d}$.
\end{itemize}

\begin{remark}\label{rem:ucont}
As a consequence of (H3) and (H5), $f$ is $p$-Lipschitz, i.e.,
\[
	\abs{f(z,\xi)-f(z,\mu)}\leq c_4 (\abs{\xi}+\abs{\mu})^{p-1} \abs{\xi-\mu} \qquad \text{for all $\xi, \mu\in \R^m$ and almost all $z\in \R^{d}$}
\]
with a constant $c_4\geq 0$.
By H\"older's inequality, this implies that
\[
	u\mapsto f\Big(\frac{\cdot}{\delta},u(\cdot)\Big),\quad L^p(\Omega_1;\R^m)\to L^1(\Omega_1),
\]
is uniformly continuous on bounded subsets of $L^p(\Omega_1;\R^m)$, also uniformly in $\delta\in (0,1]$.
\end{remark}

To transform the variational principle~\eqref{vp_beforerescaling} into one on the fixed domain $\Omega_1$, 
we apply the classical thin-film rescaling. The change of variables $y=(y',y_d)=(x',\eps x_d)$ with $u(x):=v(y)=v(x',\frac{x_d}{\eps})$ 
leads us to considering the family of functionals $(F_{\eps, \delta})$ given by
\begin{align}\label{Fepsdelta}
F_{\eps, \delta} (u)= \begin{cases}
\displaystyle\int_{\Omega_1} f\Bigl(\frac{x'}{\delta}, \frac{\eps x_d}{\delta}, u(x)\Bigr)\dd{x}, 
& \text{if $\Acal_\eps u =0$ in $\Omega_1$,}\\
\infty, & \text{otherwise,}
\end{cases}\qquad u\in L^p(\Omega_1;\R^m).
\end{align}
The rescaled differential operator becomes parameter-dependent and is  defined as
\begin{equation*}%\label{operator_A}
\Acal_\eps= \Acal' + \frac{1}{\eps}A^{(d)}\partial_d \qquad\text{with $\Acal':=\sum_{k=1}^{d-1} A^{(k)} \partial_k$} 
\end{equation*}
for $\eps>0$. 

Let us remark that any linear partial differential operator of first order
\begin{align*}
\Bcal=\sum_{k=1}^d B^{(k)}\partial_k \qquad \text{with given matrices $B^{(1)}, \ldots, B^{(d)}\in \R^{l\times m}$}
\end{align*}  
can be interpreted as a bounded linear operator $\Bcal: L^p(\Omega_1;\R^m)\to W^{-1,p}(\Omega_1;\R^l)$ by
\begin{align*}
(\Bcal u)[v]=-\int_{\Omega_1}u\cdot \Bcal^T v \dd{x}, \qquad u\in L^p(\Omega_1;\R^m),\, v\in W_0^{1,p'}(\Omega_1;\R^l).
\end{align*}
Here, $\Bcal^T=\sum_{k=1}^d(B^{(k)})^T\partial_k$.
The differential constraint $\Bcal u=0$ in $\Omega_1$ 
(or $u\in \ker_{\Omega_1}\Bcal$) is understood 
in the sense of distributions, i.e.
\begin{align*} 
-\int_{\Omega_1}u\cdot \Bcal^T \varphi \dd{x}=0\qquad \text{for all $\varphi\in C^\infty_c(\Omega_1;\R^l)$,}
\end{align*}and we denote $\Ucal_\Bcal:=\{u\in L^p(\Omega_1;\R^m): u\in \ker_{\Omega_1}\Bcal\}$.

Throughout the paper, we often work with functions on the $d$-torus $\T^d$, or, more general, on $\T^d(Q)$ for a cuboid $Q\subset \R^d$. 
Since $L^p(\T^d(Q);\R^m)=L^p(Q;\R^m)$, we always use the shorter notation $L^p(Q;\R^m)$, assuming implicitly the identification of each function with
its periodic extension to $\R^d$.
By $\Bcal u=0$ in $\T^d$ (or $u\in \ker_{\T^d}\Bcal$) for $u\in L^p(Q^d;\R^m)$, we mean that
\begin{align*}
  -\int_{Q^d} u\cdot \Bcal^T \varphi \dd{x}=0\qquad \text{for all $\varphi\in C^\infty(\T^d;\R^l)$,}
\end{align*}
where $C^\infty(\T^d)$ is the space of smooth, $Q^d$-periodic functions that are smooth also over the gluing boundaries.

\subsection{Collection of tools and results on dimension reduction} An important step towards capturing the asymptotic behavior of $(F_{\eps, \delta})$ is the characterization of 
the limit PDE constraint in \eqref{Fepsdelta} for vanishing $\eps$. Before presenting the representation result obtained in~\cite{KR14}, we  state the required hypotheses on $\Acal$ and give the
natural definition of a limit operator for $(\Acal_\eps)$ as $\eps\todown 0$.

For our analysis, we make the following assumptions on $\Acal$:
\begin{itemize}
 \item[(A1)] {\it Constant-rank operator}\\ $\rank \Abb(\eta)=r\in \N$ for all $\eta\in \R^d\setminus\{0\}$; 
 \item[(A2)] {\it Normalization}\\ $A^{(d)}=\left[\begin{array}{c}A^{(d)}_+\\ \hline 0\end{array}\right]$ with $A^{(d)}_+\in \R^{r\times m}$ and $r =\rank A^{(d)}$.
\end{itemize}

\begin{remark}
Note that (A2) is not restrictive, since $\cA$ can always be modified  
to satisfy this condition artificially by multiplying $\cA$ from the left with a suitable invertible 
matrix in $\RR^{l\times l}$, but it helps simplify the representation of the limit operator.
\end{remark}

Let us define the operator\begin{align}\label{intro:Acal0}
  \Acal_0 := \left[\begin{array}{c}A^{(d)}_+\partial_d \\ \hline \Acal'_-\end{array}\right], 
\end{align}
where $\Acal$ is decomposed into 
\begin{align}\label{splittingA}
\Acal =\Acal'+ \left[\begin{array}{c}A^{(d)}_+\partial_d\\ \hline 0\end{array}\right]= \ \left[\begin{array}{c} \Acal'_+ +A^{(d)}_+\partial_d \\ \hline \Acal'_-\end{array}\right]=\left[\begin{array}{c}\Acal_+ \\ \hline \Acal_-\end{array}\right].
\end{align}

Here are two further hypotheses on $\Acal$:
\begin{itemize}
 \item[(A3)] {\it Extension property of $\Acal_0$}\\for every $u\in \Ucal_{\Acal_0}$ there exists a sequence 
$(\bar{u}_j)_{j}\subset L^p(Q^d;\R^m)\cap \ker _{\T^d}\Acal_0$ such that $\bar{u}_j \to u$ in $L^p(\Omega_1;\R^m)$ as $j\to \infty$;
 \item[(A4)] {\it Antisymmetry relation}\\
$A^{(k_1)}(A^{(d)})^\dagger A^{(k_2)} = - A^{(k_2)} (A^{(d)})^\dagger A^{(k_1)}$ for $k_1, k_2=1, \ldots, d-1$
with $(A^{(d)})^\dagger\in \R^{m\times l}$ the Moore-Penrose pseudoinverse of $A^{(d)}$.
\end{itemize}
\begin{remark}
The extension property required in (A3) effectively only restricts $\Acal'_-$. Indeed, extending in direction of $x_d$ is trivial, 
because if $P^{(d)}$ denotes the orthogonal projection onto $\ker A^{(d)}$ in $\RR^m$, then $A^{(d)}_+\partial_d u=0$ is equivalent to $\partial_d (I-P^{(d)}) u=0$, i.e, $(I-P^{(d)}) u$ is constant in $x_d$. This property, of course, remains valid for the $1$-periodic extension in $x_d$ of $u$.
\end{remark}
For a detailed discussion of these assumptions on $\Acal$, as well
as examples for operators satisfying (A1)-(A4) we refer to~\cite[Section~2]{KR14}.

\begin{lemma}[The limit operator \boldmath{$\Acal_0$}, {\cite[Theorem~1.1, Proposition~4.1]{KR14}}]\label{lem:A0}
Under the assumptions~(A1)-(A4), $\Acal_0: L^p(\Omega_1;\R^m)\to W^{-1,p}(\Omega_1;\R^l)$ as defined in~\eqref{intro:Acal0} is the limit operator of $(\Acal_\eps)$ for $\eps\todown 0$. 

Precisely, this means that
the following two conditions are satisfied:
\begin{itemize}
 \item[(i)] Let $\eps_j\todown 0$ and $u_j\in \Ucal_{\Acal_{\eps_j}}$ ($j\in \N$) such that $u_j\weakly u$ in $L^p(\Omega_1;\R^m)$ for $u\in L^p(\Omega_1;\R^m)$.
 Then $u\in \Ucal_{\Acal_0}$.
 \item[(ii)] For every $u\in \Ucal_{\Acal_0}$ and every $\eps_j\todown 0$, 
 there exists $u_j\in \Ucal_{\Acal_{\eps_j}}$ ($j\in \N$) such that $u_j\weakly u$ in $L^p(\Omega_1;\R^m)$.
\end{itemize}
\end{lemma}
The limit operator $\Acal_0$ can be considered unique in the sense that the set of $\Acal_0$-free vector fields $\Ucal_{\Acal_0}$ is unique.
Whereas the rescaled versions $\Acal_\eps$ of a constant-rank operator $\Acal$ have again the constant-rank property (A1), 
this is in general not true for the limit operator $\Acal_0$. 

A cornerstone for any variational problem subject to differential constraints is a 
suitable projection result that allows (re-)generating admissible fields and keeps control of the projection error. 
Such a tool was first proven by Fonseca \& M\"uller in~\cite[Lemma~2.14]{FM99} under the assumption that the constraint is conveyed by a 
constant-rank operator $\Acal$. The argument is based on discrete Fourier methods and exploits that the orthogonal projection 
$\Pbb(\eta)\in \R^{m\times m}$ onto $\ker \Abb(\eta)$ for $\eta\in \R^d\setminus \{0\}$ defines a Mikhlin Fourier-multiplier. In fact, $\Pbb$
is $0$-homogeneous and continuous in view of the constant-rank property.

The fact that the PDE constraint in~\eqref{Fepsdelta} depends on the parameter $\eps$ calls for 
a refined version of the projection result with uniformly bounded constants as provided in~\cite[Theorem~2.8]{KR14}:

\begin{lemma}[Projection onto $\Acal_\eps$-free fields, %{\cite[Lemma~2.14]{FM99}}, 
{\cite[Theorem~2.8]{KR14}}]\label{theorem_projection} 
Let $1<p<\infty$ and let $\Acal$ satisfy (A1). Then, for every $\eps>0$ there exists a linear, bounded projection operator 
$\Pcal_{\Acal_\eps}: L^p(Q^d;\R^m)\to L^p(Q^d;\R^m)\cap \ker_{\T^d}\Acal_\eps$
such that for all $u\in L^p(Q^d;\R^m)$:
\begin{enumerate}
 \item[\it (i)] 
$\norm{\Pcal_{\Acal_\eps} u}_{L^p(Q^d;\R^m)} \leq c_p \norm{u}_{L^p(Q^d;\R^m)}$
with a constant $c_p > 0$ independent of $\epsilon$.
\item[\it (ii)] There exists a constant $c_p>0$ such that for all $\eps>0$,
\[
  \qquad \norm{u-\Pcal_{\Acal_\eps} u}_{L^p(Q^d;\R^m)}\leq c_p \norm{\Acal_\eps u}_{W^{-1,p}(\T^d;\R^l)}.
\]
\item[\it (iii)] Let $\eps_j\todown 0$ and let $(u_j) \subset L^p(Q^d;\R^m)$ be $p$-equiintegrable. 
Then, the sequence $(\Pcal_{\eps_j} u_j)$ is still $p$-equiintegrable.
\end{enumerate}
\end{lemma}

Concerning dimension reduction of multiple functionals on $\Acal$-free fields, a $\Gamma$-limit result
(see~\cite{DalMa93B, Bra02} for an introduction to $\Gamma$-convergence) 
for the rescaled variational problem was established in~\cite{KR14}.
\begin{theorem}[Thin-film limit, {\cite[Theorem~1.1]{KR14}}]\label{theo:dimred}
Let $f:\omega\times\R^m \to [0, \infty)$ be the restriction of a function satisfying (H0), (H3)-(H5) and assume that $\Acal$ fulfills (A1)-(A4).
For $\eps>0$ and $u\in L^p(\Omega_1;\R^m)$ let
\begin{align*}
 I_\eps(u)=\begin{cases}
            \displaystyle\int_{\Omega_1} f(x, u(x))\dd{x} & \text{if $\Acal_\eps u=0$ in $\Omega_1$,}\\
            \infty & \text{otherwise.}
           \end{cases}
\end{align*}
Then, 
\begin{align*}
      I_0(u):=\Gamma\text{-}\lim_{\eps\to 0} I_\eps(u)=\begin{cases}
            \displaystyle\int_{\Omega_1} f(x, u(x))\dd{x} & \text{if $\Acal_0 u=0$ in $\Omega_1$,}\\
            \infty & \text{otherwise,}
           \end{cases}
     \end{align*}
regarding weak convergence in $L^p(\Omega_1;\R^m)$.
\end{theorem}
The convexity of $f$ makes the proof of the lower bound trivial, while a stronger version of Lemma~\ref{lem:A0},
where weak $L^p$-convergence in (ii) is replaced with strong convergence, yields the upper bound. 
In~\cite{KR14}, upper and lower bounds on the $\Gamma$-limit of $(I_\eps)$ are given also for  nonconvex $f$. The question of whether $I_0$ is local, though, is to our knowledge still open in that case.

\subsection{Collection of results on homogenization} 
Homogenization in the context of~variational problems restricted to $\Acal$-free 
fields was first studied by Braides, Fonseca \& Leoni in~\cite{BFL00}. 
Later, this result was enhanced in~\cite{FoKroe10a}, where the use of two-scale techniques allowed for weaker assumptions on the integrand.

\begin{theorem}[Homogenization, {\cite[Theorem~1.1]{FoKroe10a}}]\label{theo:hom}
Let $f$ satisfy (H0), (H2), (H3) and $\Acal$ be of constant rank, i.e.~(A1) holds. For $\delta>0$ and $u\in L^p(\Omega_1;\R^m)$ let
\begin{align*}
 J_\delta(u)=\begin{cases}
            \displaystyle\int_{\Omega_1} f\Bigl(\frac{x}{\delta}, u(x)\Bigr)\dd{x} & \text{if $\Acal u=0$ in $\Omega_1$,}\\
            \infty & \text{otherwise.}
           \end{cases}
\end{align*}
Then, the $\Gamma$-limit with respect to weak convergence in $L^p(\Omega_1;\R^m)$ has the form
\begin{align*}
      J_0(u):=\Gamma\text{-}\lim_{\delta\to 0} J_\delta(u)=\begin{cases}
            \displaystyle\int_{\Omega_1} \fhom(u(x))\dd{x} & \text{if $\Acal u=0$ in $\Omega_1$,}\\
            \infty & \text{otherwise,}
           \end{cases}\end{align*}
where 
\begin{align}\label{multicell}
\fhom(\xi) = \liminf_{n\to \infty} \inf_{v\in \Vcal_\Acal} \dashint_{Q^d}f(ny, \xi + v(y))\dd{y} 
= \inf_{n\in \N} \inf_{v\in \Vcal_\Acal} \dashint_{Q^d}f(ny, \xi + v(y))\dd{y}
\end{align}
for $\xi\in \R^m$ with $\Vcal_\Acal:=\{v\in L^p(Q^d;\R^m): v\in \ker_{\T^d}\Acal, \int_{Q^d} v\dd{y}=0\}$.
\end{theorem}

Next, we collect some properties of the homogenization formula $\fhom$, as well as 
equivalent ways of writing~\eqref{multicell}.
\begin{remark}\label{rem:multicell}
 a) It is immediate to see that $\fhom$ inherits $p$-growth and $p$-coercivity from $f$ (compare (H3), (H4)), precisely
\begin{align*}
  0\leq \fhom(\xi)\leq c_1(1+\abs{\xi}^p) \qquad\text{and}\qquad \fhom(\xi)\geq c_2\abs{\xi}^p -c_3
\end{align*}
for every $\xi\in \R^m$.

b) If $f$ is convex in the second variable, meaning that (H5) holds (in addition to the assumptions of Theorem~\ref{theo:hom}), 
$\fhom$ is a convex function as well. This follows from the well-known fact that the $\Gamma$-limit of a family of convex functionals is again convex, 
which implies the convexity of $J_0$. From testing with constant ($\Acal$-free) fields, we then infer that $\fhom$ is convex.

c) If $f$ satisfies (H5), the infimum over 
$n$ in~\eqref{multicell} is attained at $n=1$. In other words, for $f(z, \frarg)$ convex for almost all $z\in \R^{d}$,
the multicell formula reduces to the cell formula
\begin{align}\label{cell}
 \fhom(\xi)=\inf_{v\in \Vcal_\Acal} \dashint_{Q^d}f(y, \xi + v(y))\dd{y},\qquad \xi\in \R^m.
\end{align}
This can be seen by a standard trick in homogenization theory (cf.~\cite[Lemma 4.1]{Mue87}): 
For any choice of $n$, the test function $v\in \Vcal_\Acal$ can be replaced by the $\frac{1}{n}Q^d$-periodic function
\begin{align*}
v^\sharp(y):=\frac{1}{n^d}\sum_{k\in \N^d_0, \abs{k}<n} v\Bigl(y+\frac{k}{n}\Bigr), \qquad y\in Q^d.
\end{align*}
By (H5) and a change of variables, we obtain that
\begin{align*}
\dashint_{Q^d} f(ny, \xi+v(y))\dd{y} \geq \dashint_{Q^d} f(ny, \xi+v^\sharp(y))\dd{y} = \dashint_{Q^d} f(y, \xi + \tilde{v}(y))\dd{y}
\end{align*}
for $\tilde{v}:= v^\sharp(\frac{\frarg}{n})$, which is again a function in $\Vcal_\Acal$.

d) Assuming that $f$ is convex in the second variable, one can use a similar argument as in c) to show that $\fhom$ can also be expressed in terms of a generalized multicell formula, that is
\begin{align*}
\fhom(\xi) = \inf_{n\in \N^d}\inf_{v\in \Vcal_\Acal}\dashint_{Q^d} f(n_1y_1, n_2y_2, \ldots, n_d y_d, \xi+v(y))\dd{y}, \qquad \xi\in \R^m.
\end{align*}

e) A density argument based on convolution shows that $\Vcal_\Acal$ in~\eqref{multicell} and~\eqref{cell}
can be substituted for the smaller set of admissible functions 
$V_\Acal$ defined by
\begin{align*}
V_{\Acal}= \{\textstyle v\in C^\infty(\T^d;\R^m): \Acal v=0\text{ in $Q^d$}, \int_{Q^d} v\dd{y}=0\}.
\end{align*}
Here, $\Acal v=0$ in $Q^d$ can be interpreted in the sense of classical derivatives and is equivalent to $\Acal v=0$ in $\T^d$ (in the sense of distributions) for $v\in C^\infty(\T^d;\R^m)$.
Notice that (H3) is needed here to apply Lebesgue's convergence theorem when passing to the limit. 
\end{remark}

Without further mentioning, we will always choose among these various equivalent definitions of $\fhom$ the one best suited 
for the situation at hand.

%%%%%%%%%%%%%%%%%%%%%%%%%%%%%%%%%%%%%%%%%%%%%%%%%%%%%%%%%%%%%%%%%%%
\section{Relevant scaling regimes and main results}\label{sec:regimes&results}
This article contains a discussion of the asymptotic behavior of 
the variational principle~\eqref{vp_beforerescaling} in the five scaling regimes illustrated in Fig.~\ref{fig:regimes}.

We distinguish two classes of limit processes, these are successive and simultaneous limits. 
In the successive case, one can either homogenize $(F_{\eps, \delta})$ first by letting $\delta$ tend to zero, and then perform the dimension reduction by passing to the limit in $\eps$, or proceed the other way around.
To keep the notation short, we will refer to these two regimes simply as $\delta, \eps \todown 0$ and $\eps, \delta\todown 0$, where the order of $\eps$ and $\delta$ indicates the order of the limit processes. On the other hand, for the simultaneous case, we assume that the relation between $\eps$ and $\delta$ is given by $\delta=\eps^\alpha$ with a scaling parameter $\alpha>0$, and study the asymptotic behavior of the family $(F_{\eps, \eps^\alpha})$ for $\eps \todown 0$. If $\alpha>1$, this means that the film thickness is large compared to the heterogeneities, or in other words that
we are dealing with a film with sufficiently fine heterogeneous substructure or layers. For $\alpha=1$, the film thickness and scale of heterogeneities are comparable, while for $\alpha<1$ one can think of thin film with coarse-grained heterogeneities.

Stated below, there is a summary of the main results of this work, which depend critically on the scaling regime under consideration. 

We start with the positive results leading to an explicit local representation formula for the limit behavior of~\eqref{vp_beforerescaling} both
in the successive regime $\delta, \eps\todown 0$ and for the simultaneous case with $\alpha>1$.
The proofs are given in Section~\ref{sec:deltaeps} and~\ref{sec:alpha>1}, respectively.

\begin{theorem}[Local limit for \boldmath{$\delta, \eps\todown 0$}]\label{theo:successive_local}
If $f$ satisfies (H0), (H2)--(H5), and $\Acal$ meets the hypotheses (A1)-(A4), then
\begin{align*}
\Gamma\text{-}\lim_{\eps\to 0}[\Gamma\text{-}\lim_{\delta \to 0} F_{\eps, \delta}(u)] = \begin{cases}\int_{\Omega_1} \fhom (u)\dd{x} & \text{if $\Acal_0 u=0$ in $\Omega_1$,} \\ \infty & \text{otherwise}\end{cases}=:F_0(u)
\end{align*}
for $u\in L^p(\Omega_1;\R^m)$. Both $\Gamma$-limits in the definition of $F_0$ are taken with respect to weak convergence in $L^p$.
\end{theorem}

The same limit functional describes the asymptotics of $(F_{\eps, \delta})$ in the related simultaneous regime.
\begin{theorem}[Local limit for \boldmath{$\alpha>1$}]\label{theo:simultaneous_local}
Let $\alpha>1$. If $f$ satisfies (H0)--(H5), and $\Acal$ meets the hypotheses (A1) and (A2), then
\begin{align*}
\Gamma\text{-}\lim_{\eps\to 0} F_{\eps, \eps^\alpha}=F_0,
\end{align*}
where the $\Gamma$-limit  is understood with respect to weak convergence in $L^p(\Omega_1;\R^m)$.
\end{theorem}

\begin{figure}[t]
\begin{tikzpicture}
\node at (-0.1,0) {$F_{\eps, \delta}$};
\node at (5,0) {$F_{\eps}$};
\node at (0,-5) {$F_\delta$};
\node at (5,-5) {$F$};
\draw[->, thick] (-2,0.5)-- node[above, rotate=90]{dimension reduction} (-2, -6);
\draw[->, thick] (-1, -6) -- node[below]{homogenization} (6,-6);
\draw[->, densely dashed] (0.25,-0.25) --(4.75,-4.75);
\draw[->]  (0.25,0) --node[above] {$\delta\todown 0$} (4.75,0);
\draw[->] (0,-0.25) --node[left] {$\eps\todown 0$} 
(0, -4.75);
\draw[->, densely dashed] (0.25,-5) --node[below] {$\delta\todown 0$}(4.75,-5);
\draw[->] (5,-0.25) -- node[right] {$\eps\todown 0$}(5,-4.75);
\draw[->, densely dashed] (0.125,-0.25) arc [radius= 4.6, start angle=180, end angle=270];
\draw[->] (0.25,-0.125) arc [radius= 4.6, start angle=90, end angle=0];
\node at (3.7,-1) {$\alpha>1$};
\node at (1,-3.7) {$\alpha<1$};
\node at (3.2,-2.5) {$\alpha=1$};
\end{tikzpicture}

\caption{Overview of scaling regimes for $\eps>0$ and $\delta>0$, with solid and dashed lines indicating local limit processes and the possibility of nonlocal effects, respectively.\label{fig:regimes}}
\end{figure}
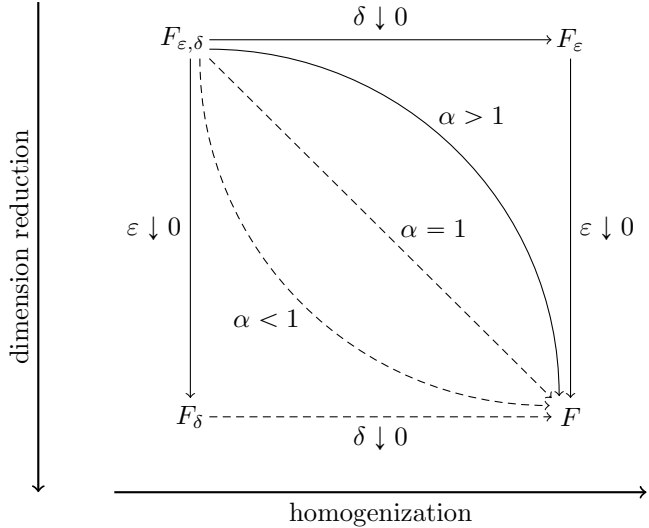

The following remark is about two relevant special cases of heterogeneities and their implications on the form of the homogenization formula.
\begin{remark}\label{lem:properties_fhom}
Let the assumptions (H0), (H2), and (H5) hold.

a) If $f$ is constant in $z_d$, i.e.~$f(z,\xi)=f(z', \xi)$ for $z\in \R^d$ and $\xi\in \R^m$, 
then
\begin{align*}
\fhom=\fhomprime
\end{align*} 
for all $\eps>0$, where
\begin{align*}
\fhomprime(\xi) := \inf_{v\in \Vcal_{\Acal'}} \dashint_{Q^{d-1}} f(y', \xi+v(y'))\dd{y'} 
\end{align*}
with $\Vcal_{\Acal'}=\{v\in L^p(Q^{d-1};\R^m): v\in \ker_{\T^{d-1}}\Acal', \int_{Q^{d-1}}v\dd{y'}=0\}$. Choosing $f$ independent of the variable $z_d$ corresponds to the assumption that the film is homogeneous in thickness direction.

Since $\Vcal_{\Acal'}\subset \Vcal_{\Acal}$ after identification of functions in $\Vcal_{\Acal'}$ with their constant extensions in $y_d$-direction, the estimate $\fhomprime\geq \fhom$ is immediate.

For the proof of the converse implication, fix $\delta>0$. By Remark~\ref{rem:multicell}\,e) we can choose $v\in V_{\Acal}$ such that
\begin{align*}
\fhom(\xi)\geq \dashint_{Q^d} f(y', \xi +v(y))\dd{y} - \delta.
\end{align*}
Exploiting the convexity of $f$ in the second argument by applying Jensen's inequality with respect to the $y_d$-variable leads to
\begin{align}\label{est8}
\fhom(\xi)\geq \dashint_{Q^{d-1}} f(y', \xi + \textstyle\int_0^1v(y', y_d)\dd{y_d})\dd{y'}-\delta= \displaystyle\dashint_{Q^{d-1}} f(y', \xi +w(y'))\dd{y'}-\delta,
\end{align}
where $w:=\int_{0}^1 v(\frarg, y_d)\dd{y_d}\in C^\infty(\T^{d-1};\R^m)$ with $\int_{Q^{d-1}}w\dd{y'}=\int_{Q^d} v\dd{y}=0$.  Besides, it holds that  $\Acal' w=0$ in $\T^{d-1}$ (in the sense of distributions) or, equivalently $\Acal' w=0$ in $Q^{d-1}$ pointwise. To see the latter, we argue that 
\begin{align*}
\Acal' w(y')&= \int_0^1 \Acal' v(y', y_d)\dd{y_d} =\int_0^1 \Acal v(y', y_d) -A^{(d)} \partial_d v(y',y_d)\dd{y_d}\\ &=-A^{(d)}\int_0^1 \partial_d v(y', y_d)\dd{y_d} =0
\end{align*}
for $y'\in Q^{d-1}$, in view of $\Acal v=0$ in $Q^d$ and the periodicity of $v\in C^\infty(\T^d;\R^m)$. 
Thus, $\fhom\geq \fhomprime$ follows from~\eqref{est8}, considering that $\delta>0$ was arbitrary.

b) Another special case is the situation when the material composition of the film can only vary in thickness direction and is homogeneous along the film. If $f$ is constant in $z'$, i.e.~$f(z, \xi)=f(z_d, \xi)$ for $z\in \R^d$ and $\xi\in \R^m$, then
\begin{align}\label{equ1}
\fhom(\xi)= f_{\rm hom}^{A^{(d)}}(\xi):= \inf_{v\in \Vcal_{A^{(d)}}} \int_0^1 f(y_d, \xi+v(y_d))\dd{y_d} 
\end{align}
with $\Vcal_{A^{(d)}}=\{v\in L^p(0,1;\R^m) : A^{(d)} v={\rm const.}, \int_{0}^1v\dd{y_d}=0\}$.
The proof of~\eqref{equ1} is similar to a) and even easier, just observe that for $w:=\int_{Q^{d-1}}v(y', \frarg)\dd{y'}$ with $v\in V_\Acal$ we have that $A^{(d)}\partial_d w(y_d) = \int_{Q^{d-1}} A^{(d)}\partial_d v(y', y_d)\dd{y'}=-\int_{Q^{d-1}} \Acal' v(y', y_d)\dd{y'}=0$ for $y_d\in (0,1)$. Here, we have use again the periodicity of $v$.
\end{remark}

For the remaining regimes, we prove that a local limit functional cannot exist.
Here, we state only the result for the simultaneous regimes with $\alpha\leq 1$. An analogous statement holds for the successive limits $\eps, \delta\todown 0$, see Section~\ref{sec:epsdelta}.

To make this more precise, for a sequence $\eps_j\todown 0$ and an open set $D\subset \Omega_1$, we set
\begin{align*}
F_0^-(u;D)=\inf \{\liminf_{j\to \infty} \textstyle \int_D f (\frac{x}{\eps_j^\alpha}, \frac{x_d}{\eps_j^{\alpha-1}}, u_j(x))\dd{x}: u_j \in \Ucal_{\Acal_{\eps_j}} (j\in \N), u_j\weakly u \text{ in $L^p(\Omega_1; \R^m)$}\},
\end{align*}
for which we obtain the following result.

\begin{theorem}[Nonlocal example for \boldmath{$\alpha\leq 1$}]\label{theo:simultaneous_nonlocal}
Let $\alpha\leq 1$, $\Omega_1=Q^d$, and $\eps_j\todown 0$ be the sequence with $\eps_j^\alpha=\frac{1}{j}$. 
If $\Acal$ has the properties (A1) and (A2) and is nontrivial in the sense that the kernel of its symbol, $\ker \Abb(\eta)$, is not constant in $\eta\neq 0$ and $\ker\Abb(e_d)\neq \{0\}$ (in particular, $\cA\neq 0$ and $\cA\neq \nabla$), 
then there exist $f:\R^{d}\times \R^m \to \R$ satisfying (H0)-(H5) and constant in $z_d$, as well as a $u_0\in \Ucal_{\Acal_0}$, such 
that the set function $F_0^-(u_0; \frarg)$ is not additive.
\end{theorem}
An explicit counterexample to the additivity of $F_0^-(u_0;\frarg)$ is constructed in Section~\ref{sec:alpha<1}.
Hence, $\Gamma$-$\liminf_{\eps\to0} F_{\eps, \eps^\alpha}$ is nonlocal for $\alpha\leq 1$, which carries over to the corresponding $\Gamma$-limit (if it exists). 

%%%%%%%%%%%%%%%%%%%%%%%%%%%%%%%%%%%%%%%%%%%%%%%%%%%%%%%%%%%%%%%%%%%%%%%%%

\section{Thin-film limit of the homogenized problem}\label{sec:deltaeps}
This paragraph is devoted to the proof of Theorem~\ref{theo:successive_local}, or in other words to the identification of the successive $\Gamma$-limit
\begin{align}\label{consecutiveGamma}
\Gamma\text{-}\lim_{\eps\to 0}[\Gamma\text{-}\lim_{\delta\to 0} F_{\eps, \delta}],
\end{align}  
see~\eqref{Fepsdelta} for the definition of the family of functionals $(F_{\eps, \delta})$.

We join the results of Theorem~\ref{theo:hom} on homogenization and Theorem~\ref{theo:dimred} on dimension reduction in the context of $\Acal$-free fields to derive a local characterization for\eqref{consecutiveGamma}.

\begin{proof}[Proof of Theorem~\ref{theo:successive_local}]
Let $\eps>0$ be fixed. We rewrite $F_{\eps, \delta}$ with $\delta>0$ as
\begin{align*}
F_{\eps, \delta}(u) = 
\begin{cases}\int_{\Omega_1} f_\eps (\frac{x}{\delta}, u(x))\dd{x} & \text{if $\Acal_\eps u=0$ in $\Omega_1$,} \\ \infty & \text{else,}\end{cases}\qquad u\in L^p(\Omega_1;\R^m),
\end{align*}
where $f_\eps(z, \xi):=f(z', \eps z_d, \xi)$ for $z\in \R^d$ and $\xi\in \R^m$. Observe that $f_\eps(\frarg, \xi)$ is $E_\eps$-periodic for every $\xi\in \R^m$ with $E_\eps:=Q^{d-1}\times (0, 1/\eps)$. Moreover, $f_\eps$ fulfills (H0), (H3), and (H4). 
Hence, $f_\eps$ satisfies the requirements of Theorem~\ref{theo:hom} with the periodic cell $Q^d$ replaced by $E_\eps$. Since $f_\eps$ is convex in the second variable, the homogenized integrand can be expressed using the cell formula (compare Remark~\ref{rem:multicell}\,b)), so we infer from Theorem~\ref{theo:hom} that
\begin{align*}
F_{\eps}(u):=\Gamma\text{-}\lim_{\delta \to 0} F_{\eps, \delta}(u) = 
\begin{cases}\int_{\Omega_1}  f_{\rm hom}^{\eps}(u)\dd{x} & \text{if $\Acal_\eps u=0$ in $\Omega_1$,} \\ \infty & \text{else}\end{cases}\end{align*}
with
\begin{align*}
f_{\rm hom}^\eps(\xi):= \inf_{w\in \Vcal_\eps} \dashint_{E_\eps} f_\eps(z, \xi +w(z))\dd{z}
\end{align*} for $\xi\in \R^m$ 
and $\Vcal_\eps:=\{w\in L^p(E_\eps;\R^m): w\in \ker_{\T^d(E_\eps)}\Acal_\eps , \int_{E_\eps} w\dd{z}=0\}$.
After a scaling argument, one finds that 
\begin{align}\label{fepshom=fhom}
f_{\rm hom}^\eps=\fhom\qquad\text{ for all $\eps>0$}.
\end{align} 
Indeed, if $w(z)=v(y)$ for $z\in E_\eps$ with $y=(z', \eps z_d)\in Q^d$, then $v\in \Vcal_{\Acal}$ if and only if $w\in \Vcal_\eps$, and~\eqref{fepshom=fhom} follows from
\begin{align*}
f_{\rm hom}^\eps(\xi)= \inf_{w\in \Vcal_{\eps}}\dashint_{E_\eps} f(z', \eps z_d, \xi+w(z))\dd{z} = \inf_{v\in \Vcal_\Acal} \dashint_{Q^d}f(y, \xi+v(y))\dd{y}=\fhom(\xi)
\end{align*}
for $\xi\in \R^m$.

Finally, in view of~\eqref{fepshom=fhom} the characterization of the thin-film limit $\Gamma$-$\lim_{\eps\to0}F_\eps$ is an immediate consequence of Theorem~\ref{theo:dimred}, as $\fhom$ is convex and has $p$-growth and $p$-coercivity by Remark~\ref{rem:multicell}\,a) and b).
\end{proof}
%%%%%%%%%%%%%%%%%%%%%%%%%%%%%%%%%%% Simultaneous limits I: Thin films with fine heterogeneity ($\alpha>1$) %%%%%%%%%%%%%%%%%%%%%%%%

\section{Thin films with fine heterogeneity ($\alpha>1$)}\label{sec:alpha>1}
The ultimate goal of this section is the proof of Theorem~\ref{theo:simultaneous_local} on the characterization of $\Gamma$-$\lim_{\eps \to 0} F_{\eps, \eps^\alpha}$ for $\alpha>1$. This will be a direct consequence of joining the results on the lower bound and the construction of a suitable recovery sequence, which are provided in Proposition~\ref{prop:lowerbound} and Proposition~\ref{prop:upperboundA}, respectively.

\subsection{Lower bound}
The proof of the lower bound relies on arguments from multiscale convergence, see e.g.~\cite{AllaireBriane96}, which is a natural generalization of the concept of two-scale convergence introduced by Nguetseng~\cite{Nguetseng1989} and Allaire~\cite{Allaire}.
The following definition is a special type of three-scale convergence adapted to the context of this paper.

\begin{definition}[Reduced weak three-scale convergence]
Let  $(u_\eps)\subset L^p(\Omega_1;\R^m)$,  $w\in L^p(\Omega_1\times Q^d;\R^m)$, and $\alpha>1$. We say that
$(u_\eps)$ weakly three-scale converges to $w$ in a reduced sense, or $u_\eps \alphaweakly w$ in $L^p(\Omega_1;\R^m)$,
if
\begin{align*}
\lim_{\eps\to 0} \int_{\Omega_1} u_\eps(x) \cdot \varphi\Bigl(x, \frac{x'}{\eps^\alpha}, \frac{x_d}{\eps^{\alpha-1}}\Bigr)\dd{x} = \int_{\Omega_1}\int_{Q^d}w(x, y)\varphi(x,y)\dd{y}\dd{x}
\end{align*}
for every $\varphi\in L^{p'}(\Omega_1;C^\infty(\T^d;\R^m))$.
\end{definition}
Bounded sequences are compact regarding reduced weak three-scale convergence, as stated in the following. 
The proof requires a slight modification of~\cite[Theorem~1.2]{Allaire} for the classical weak two-scale convergence, compare also~\cite[Theorem~2.4]{AllaireBriane96}. 
\begin{lemma}[Compactness]\label{lem:compactness_threescale}
Let $\alpha>1$ and $(u_\eps)$ be a uniformly bounded sequence in $L^p(\Omega_1;\R^m)$. Then there exists a subsequence (not relabeled) of $(u_\eps)$ and $w\in L^p(\Omega_1\times Q^d;\R^m)$ such that 
\begin{align*}
u_{\eps}\alphaweakly w \qquad \text{in $L^p(\Omega_1;\R^m)$.}
\end{align*}
\end{lemma}
The next lemma is a straightforward adaption of~\cite[Lemma~4.4]{FoZa03}.
\begin{lemma}[Lower semicontinuity]\label{lem:lsc_threescale}
\hspace{-0.1cm}Let $f$ satisfy (H0)-(H5) and let $(u_\eps)\subset L^p(\Omega_1;\R^m)$ be such that $u_\eps \alphaweakly w$ in $L^p(\Omega_1;\R^m)$ with $\alpha>1$ and $w\in L^p(\Omega_1\times Q^{d};\R^m)$.
Then,
\begin{align*}
\liminf_{\eps\to 0}\int_{\Omega_1}f\Bigl(\frac{x'}{\eps^\alpha}, \frac{x_d}{\eps^{\alpha-1}}, u_\eps(x)\Bigr)\dd{x} &\geq \int_{\Omega_1} \int_{Q^{d}} f(y, w(x,y))\dd{y}\dd{x}.
\end{align*}
\end{lemma}

Next, we derive necessary conditions on the asymptotic behavior of admissible fields, i.e.~$\Acal_\eps$-free fields, regarding reduced weak three-scale convergence.
\begin{lemma}[Properties of reduced weak three-scale limits of $\Acal_\eps$-free sequences]\label{lem:char_threescale}
Let $(A1)$ and $(A2)$ hold, and let $\alpha>1$. Suppose $u_\eps \in \Ucal_{\Acal_\eps}$ $(\eps>0)$ is such that $u_\eps \alphaweakly w$ in $L^p(\Omega_1;\R^m)$  for $w\in L^p(\Omega_1\times Q^d; \R^m)$. 
Then, 
\begin{align*}
\Acal_0^x \bar{w}=0 \text{ in $\Omega_1$}\qquad \text{and} \qquad \Acal^y w(x, \frarg)=0\text{ in $\T^d$ for almost every $x\in \Omega_1$},
\end{align*}
where $\bar{w}(x)=\int_{Q^d} w(x,y)\dd{y}$ for $x\in \Omega_1$. 
\end{lemma}

\begin{proof} 
As $u_\eps\alphaweakly w$ in $L^p(\Omega_1;\R^m)$ particularly implies that $u_\eps \weakly \bar{w}$ in $L^p(\Omega_1;\R^m)$, we conclude that $\bar{w}\in \Ucal_{\Acal_0}$. Recall that weak limits of $\Acal_\eps$-free sequences are $\Acal_0$-free by~Lemma~\ref{lem:A0}.
 
Let $\psi\in C_c^\infty(\Omega_1)$ and $\phi\in C^\infty(\T^d;\R^l)$ and set $\varphi_\eps(x):=\psi(x)\phi(\frac{x'}{\eps^\alpha}, \frac{x_d}{\eps^{\alpha-1}})$ for $x\in \Omega_1$. From the fact that $\Acal_{\eps}u_\eps=0$ in $\Omega_1$ for every $\eps>0$, we infer that
\begin{align*}
0&= \int_{\Omega_1} u_\eps(x) \cdot \eps^\alpha \Acal^T_\eps \varphi_\eps(x)\dd{x}\nonumber\\ &=
\int_{\Omega_1} u_\eps(x)\cdot \eps^\alpha \Big[\frac{1}{\eps^\alpha}(\Acal')^T\phi\Bigl(\frac{x'}{\eps^\alpha}, \frac{x_d}{\eps^{\alpha-1}}\Bigr) + \frac{1}{\eps}\frac{1}{\eps^{\alpha-1}} (A^{(d)})^T\partial_d\phi\Bigl(\frac{x'}{\eps^\alpha}, \frac{x_d}{\eps^{\alpha-1}}\Bigr)\Big] \psi(x) \\
&\qquad\quad+ u_\eps(x)\cdot \eps^\alpha \Big[\sum_{k=1}^{d-1}(A^{(k)})^T \phi\Bigl(\frac{x'}{\eps^\alpha}, \frac{x_d}{\eps^{\alpha-1}}\Bigr)\partial_k \psi(x) + \frac{1}{\eps} (A^{(d)})^T \phi\Bigl(\frac{x'}{\eps^\alpha}, \frac{x_d}{\eps^{\alpha-1}}\Bigr) \partial_d \psi(x)\Big]\dd{x}.\nonumber
\end{align*}
By letting $\eps$ tend to zero in the above equality, we derive from the reduced weak three-scale convergence of $(u_\eps)$ that
\begin{align*}
0&=\lim_{\eps\to 0} \int_{\Omega_1} u_\eps(x) \cdot \Acal^T\phi\Bigl(\frac{x'}{\eps^\alpha}, \frac{x_d}{\eps^{\alpha-1}}\Bigr)\psi(x)\dd{x}=\int_{\Omega_1}\int_{Q^d} w(x,y) \cdot \Acal^T\phi(y)\psi(x)
\dd{y}\dd{x}\\ &= \int_{\Omega_1}\Bigr(\int_{Q^d} w(x,y) \cdot \Acal^T\phi(y)\dd{y}\Bigl) \psi(x)\dd{x}.
\end{align*}
Since $\psi\in C_c^\infty(\Omega_1)$ is an arbitrary test function, one obtains $\int_{Q^d} w(x,y) \cdot \Acal^T\phi(y)\dd{y}=0$ for almost very~$x\in \Omega_1$ and all $\phi\in C^\infty(\T^d;\R^l)$, meaning that 
\begin{align*}
\Acal^y w(x, \frarg)=0 \quad\text{ in $\T^d$ for almost every $x\in \Omega_1$.}
\end{align*}
\end{proof}

Building on the previous lemma, we can now prove the desired liminf-inequality.
\begin{proposition}[Lower bound]\label{prop:lowerbound}
Let $\alpha>1$, and assume that (H0)--(H5) and (A1)-(A2) hold. If $u_\eps\in \Ucal_{\Acal_\eps}$ $(\eps>0)$ such that $u_\eps\weakly u$ in $L^p(\Omega_1;\R^m)$, then $u\in \Ucal_{\Acal_0}$ and
\begin{align*}
\liminf_{\eps\to 0}\int_{\Omega_1} f\Bigl(\frac{x'}{\eps^\alpha}, \frac{x_d}{\eps^{\alpha-1}},u_\eps(x)\Bigr)\dd{x}\geq \int_{\Omega_1} \fhom (u(x))\dd{x}.
\end{align*}
\end{proposition}
\begin{proof}
Let $w\in L^p(\Omega_1\times Q^d;\R^m)$ denote the reduced weak three-scale limit of the sequence $(u_\eps)$, which exists, possibly after passing to a subsequence,  according to Lemma~\ref{lem:compactness_threescale}. Then, Lemma~\ref{lem:lsc_threescale} yields
\begin{align*}
&\liminf_{\eps\to 0} \int_{\Omega_1} f\Bigl(\frac{x'}{\eps^\alpha},\frac{x_d}{\eps^{\alpha-1}}, u_\eps(x)\Bigr)\dd{x} \geq \int_{\Omega_1}\int_{Q^d} f(y, w(x,y))\dd{y}\dd{x} \\ 
&\qquad = \int_{\Omega_1}\dashint_{Q^{d}} f(y, u(x) + (w(x,y)-u(x)))\dd{y}\dd{x}
\geq  \int_{\Omega_1} \fhom(u(x)) \dd{x} .
\end{align*}
In the last estimate, we have used that $w(x, \frarg)-u(x)\in \ker_{\T^{d}}\Acal$ for almost every $x\in \Omega_1$ by Lemma~\ref{lem:char_threescale}, as well as $\int_{Q^d}w(x,y)\dd{y}=\bar{w}(x)=u(x)$ for $x\in \Omega_1$. Thus, $w(x, \frarg) -u(x)\in \Vcal_{\Acal}$ for almost every $x\in \Omega_1$.
\end{proof}

\begin{remark}[A lower bound for \boldmath{$\alpha>0$}]\label{rem:upperbound_twoscale}
a) For $\alpha>0$, a statement and proof along the lines of Lemma~\ref{lem:char_threescale} reveals that classical weak two-scale limits (with respect to the scales $x$ and $x/\eps^\alpha$) of $\Acal_\eps$-free fields are necessarily $\Acal_0$-free, both in the slow and the fast variable.
Precisely, let $u_\eps \in \Ucal_{\Acal_\eps}$ $(\eps>0)$ such that $u_\eps\stackrel{2\text{-}s^\alpha}{\weakly} w$ in $L^p(\Omega_1;\R^m)$ with $w\in L^p(\Omega_1\times Q^d; \R^m)$, i.e.~
\begin{align*}
\lim_{\eps\to 0} \int_{\Omega_1} u_\eps(x) \cdot \varphi\Bigl(x, \frac{x}{\eps^\alpha}\Bigr)\dd{x} = \int_{\Omega_1}\int_{Q^d}w(x, y)\varphi(x,y)\dd{y}\dd{x}
\end{align*}
for every $\varphi\in L^{p'}(\Omega_1;C^\infty(\T^d;\R^m))$,
then
\begin{align}\label{A0_fastslow}
\Acal_0^x \bar{w}=0 \text{ in $\Omega_1$}\qquad \text{and} \qquad \Acal_0^y w(x, \frarg)=0\text{ in $\T^d$ for almost every $x\in \Omega_1$}.
\end{align}
b) As a consequence of a), if (A1)-(A2) and (H0)-(H5) hold, and $f$ is constant in $z_d$~(compare~Remark~\ref{lem:properties_fhom}\,a)), then 
\begin{align}\label{lowerbound2}
\liminf_{\eps\to 0}\int_{\Omega_1} f\Bigl(\frac{x'}{\eps^\alpha}, u_\eps(x)\Bigr)\dd{x}\geq \int_{\Omega_1} \fhomzero (u(x))\dd{x}
\end{align}
for all sequences $u_\eps\in \Ucal_{\Acal_\eps}$ $(\eps>0)$ with $u_\eps\weakly u$ in $L^p(\Omega_1;\R^m)$, where 
\begin{align*}
\fhomzero(\xi) = \inf_{v\in \Vcal_{\Acal_0}} \dashint_{Q^{d}} f(y, \xi+v(y))\dd{y}, \qquad \xi\in \R^m,
\end{align*}
and $\Vcal_{\Acal_0}=\{v\in L^p(Q^{d};\R^m): v\in \ker_{\T^{d}}\Acal_0, \int_{Q^{d}}v\dd{y}=0\}$.  To obtain this lower bound one can proceed as in the proof of Proposition~\ref{prop:lowerbound}, replacing the lower semicontinuity result in~Lemma~\ref{lem:lsc_threescale} with a two-scale version of~\cite[Lemma~4.4]{FoZa03}, which again is essentially a generalization of~\cite[Theorem~1.8]{Allaire}.

We observe that
\begin{align*}
\fhom=\fhomprime \geq \fhomzero.
\end{align*}
In view of Remark~\ref{lem:properties_fhom}\,a), this follows directly from $\Vcal_{\Acal'}\subset \Vcal_{\Acal_0}$ by identifying $L^p(Q^{d-1};\R^m)$-functions with elements in $L^p(Q^d;\R^m)$ through constant extension in $y_d$-direction. 
The question whether one even has $\fhom=\fhomprime=\fhomzero$ is open at this point. However, if equality holds, this is necessarily attributed to the convexity of $f$. In the homogeneous case with a general $f:\R^m\to \R$ and $\Acal=\curl$, one finds that $\fhom=\Qcal f$ and $\fhomzero=\Qcal_{\curl_0}f$, where $\Qcal f$ is the classical quasiconvex envelope of $f$ and $\Qcal_{\curl_0}f$ coincides with the cross-quasiconvex envelope of $f$~(compare~\cite[Section~5]{KR14}), which can be strictly smaller than $\Qcal f$ as pointed out for example in \cite{LR00}.
 
 By the way, if $A^{(d)}=A^{(d)}_+$ as in $\Acal=\diverg$, one can show that $\fhomzero$ is identical with the unconstraint homogenization formula $f_{\rm hom}^0$ with $\Acal=0$.

c) We remark that~\eqref{A0_fastslow} and~\eqref{lowerbound2} are still true if the sequences $u_\eps \in \Ucal_{\Acal_\eps}$ $(\eps>0)$ are replaced by $\Acal_0$-free fields, so by $(u_\eps)\subset \Ucal_{\Acal_0}$.
This can be viewed as a consequence of~\cite[Proposition~2.10]{FoKroe10a} applied with $\Acal=\Acal_0$. Notice that the argument does not use the constant-rank property of the operator $\Acal$.
\end{remark}

%%%%%%%%%%%%%%%%%%%%%%%%%%%%%%%%%%%%%%%%%%%%%%%%%%%%%%%%%%%%%%%%%%%%%%%%%%%%%%%%%%%%%%%%%%%%%%%%%%%%%%%%%%%%%%%%%%%%%%%%%%%%%%%%%%%%%%%%%
\subsection{Upper bound I: The case of constant weak limits}

We first prove the following special case of the upper bound.

\begin{proposition}\label{prop:ubA-constant}
Let $\alpha>1$ and assume that (H0), (H2) and (H3) hold. Then for every 
$\xi\in \RR^m$ and every $\eps_j\todown 0$, there is a $p$-equiintegrable sequence 
$(v_j)\subset L^p(\Omega_1;\R^m)$ with $v_j\in \ker_{\Omega_1}\Acal_{\eps_j}$ for $j\in \N$
such that $v_j\weakly 0$ in $L^p(\Omega_1;\R^m)$ for $j\to \infty$ and
\begin{align*}
\limsup_{j\to \infty} \int_{\Omega_1}f\Bigl(\frac{x'}{\eps_j^\alpha}, \frac{x_d}{\eps_j^{\alpha-1}}, \xi+ v_j(x)\Bigr)\dd{x} =
\int_{\Omega_1} \fhom(\xi)\dd{x}
\end{align*}
where $\fhom$ is defined by \eqref{cell}.
\end{proposition}

\begin{remark}\label{rem:ubA-constant}
By construction, $v_j$ is actually defined on all of $\RR^d$, $\eps_j^\alpha$-periodic in the first $d-1$ variables and $\eps_j^{\alpha-1}$-periodic in the last. Moreover, it satisfies $\cA_{\eps_j}v_j=0$ on $\RR^d$. In particular,
$\cA_{\eps_j}v_j=0$ on $\TT^{d-1}\times (0,1)$ whenever $\eps_j^{-\alpha}\in \NN$.
\end{remark}
\begin{proof}[Proof of Proposition~\ref{prop:ubA-constant}]
For $v\in \Vcal_{\Acal}$ and $j\in \N$ we set
\begin{align*}
v_j(x)=v\Bigl(\frac{x'}{\eps_j^\alpha}, \frac{x_d}{\eps_j^{\alpha-1}}\Bigr), \qquad x\in \Omega_1.
\end{align*}
Then,
\begin{align*}
\Acal_{\eps_j}v_{j}(x)=\frac{1}{\eps_j^\alpha}\Acal v\Bigl(\frac{x'}{\eps_j^\alpha}, \frac{x_d}{\eps_j^{\alpha-1}}\Bigr)=0\qquad \text{in $\Omega_1$,}
\end{align*}
and $v_j\weakly \dashint_{Q^d}v\dd{y}=0$ in $L^p(\Omega_1;\R^m)$. 
Moreover, $(v_j)$ is clearly $p$-equiintegrable. 
Covering $\Omega_1$ up to a thin layer near the boundary with cuboids 
of the form
\[
	R_{j}(i):=[\eps_j^{\alpha}(i'+Q^{d-1})]\times [\eps_j^{\alpha-1}(i_d,i_d+1)]\subset \RR^{d-1}\times \RR,
\]
using only those $i=(i',i_d)\in \ZZ^d$ such that $R_{j}(i)\subset \Omega_1$, i.e.,
\[
	i\in I_j:=\{z\in \ZZ^d\mid R_{j}(z)\subset \Omega_1 \},
\]
we have that 
\[
	\absB{\Omega_1\setminus \bigcup_{i\in I_j} R_j(i)}\to 0
	\quad\text{and}\quad 
	\eps_j^{(d-1)\alpha +\alpha-1} \# I_j=\sum_{i\in I_j}\absb{R_{j}(i)}  \to \abs{\Omega_1}
\]
as $j\to \infty$. Therefore, by equiintegrability in conjunction with (H2), the periodicity of $f$, and a change of variables,
\begin{align*}
&\limsup_{j\to \infty}\int_{\Omega_1} f\Bigl(\frac{x'}{\eps_j^\alpha},\frac{x_d}{\eps_j^{\alpha-1}}, \xi+v_j(x)\Bigr)\dd{x}\\
&\qquad=\limsup_{j\to \infty} \sum_{i\in I_j} \int_{R_j(i)} 
f\Bigl(\frac{x'}{\eps_j^\alpha}, \frac{x_d}{\eps_j^{\alpha-1}}, \xi+v\Bigl(\frac{x'}{\eps_j^\alpha}, \frac{x_d}{\eps_j^{\alpha-1}}\Bigr)\Bigr)\dd{x} \\
&\qquad=\limsup_{j\to \infty}  \eps_j^{(d-1)\alpha+\alpha-1} \sum_{i\in I_j}  \int_{Q^{d}} 
f\Bigl(y', y_d, \xi+v(y', y_d)\Bigr)\dd{y} \\
&\qquad =\abs{\Omega_1}\int_{Q^d} f(y, \xi+v(y))\dd{y}.
\end{align*}
\end{proof}

%%%%%%%%%%%%%%%%%%%%%%%%%%%%%%%%%%%%%%%%%%%%%%%%%%%%%%%%%%%%%%%%%%%%%%%%%%%%%%%%%%%%%%%%%%%%%%%%%%%%%%%%%%%%%%%%%%%%%%%%%%%%%%%%%%%%%%%%%%%
\subsection{Upper bound II: Localization}

As the next step towards an upper bound in the general case, it is natural to consider weak limits which are piecewise constant. To handle this, we essentially have to find a method which allows us to glue together different recovery sequences, in the simplest case with a transition layer in the neighborhood of a fixed plane. 
The main difficultly arises from possible jumps across a plane of the form $\{x_d=c\}$, because the term $\frac{1}{\eps_j}A^{(d)}\partial_d$ in our differential constraint becomes hard to control, if we introduce artificial transitions via a cut-off in direction of $x_d$ -- in the worst case, we are unable to glue in an $\cA_{\eps_j}$-free way. The key observation here is that this problem can be overcome provided that the sequences we are trying to glue together oscillate fast enough in direction of $x'$ (compared to $\eps_j$).
\begin{lemma}[Localization in \boldmath{$x_d$}, given fast oscillation in \boldmath{$x'$}]\label{lem:localize}
Let $\eps_j\todown 0$, and let $(v^\sharp_j)\subset L^p(Q^{d};\RR^m)$ be a bounded sequence with
$\cA_{\eps_j} v^\sharp_j=0$ in $\TT^{d-1} \times (0,1)$. 
In addition, suppose that for each $j\in \N$,
\[
	\text{$v^\sharp_j$ is $\tau_j$-periodic in the first $d-1$ variables}
\]
with a sequence $\tau_j\todown 0$ such that $1/\tau_j \in \NN$ and $\frac{\tau_j}{\eps_j}\todown 0$. 
Then, for every function $\eta:Q^d\to \R$ defined by $\eta(x)=\eta_d(x_d)$ for $x\in Q^d$ with $\eta_d\in C_c^1((0,1);[0,1])$,
\begin{align}\label{cutoff}
	\normb{\cA_{\eps_j}[\eta (v^\sharp_j-\bar{v}_j)]}_{W^{-1,p}(\TT^{d};\RR^l)}
	\To 0\qquad \text{as $j\to \infty$},
\end{align}
where $\bar{v}_j$ is the cell average of $v^\sharp_j$ in $x'$, i.e.,
$$
\begin{alignedat}{2}
	&&&\bar{v}_j(x):=\int_{Q^{d-1}} v^\sharp_j(\tau_j y',x_d) \dd{y'}, \qquad x\in Q^d.
\end{alignedat}
$$
\end{lemma}

\begin{remark}\label{rem:proplocal}
a) The main point of the lemma is to control the influence of the cut-off function $\eta$ in film thickness direction creating a transition between $v^\sharp_j-\bar{v}_j$ and zero. As~\eqref{cutoff} shows, we do not move too far away from the class of $\cA_{\eps_j}$-free fields, which later allows us to project back with an acceptable error.

b) 
The requirement that $1/\tau_j\in \NN$ is not really necessary, but convenient
because $1$-periodicity and $\tau_j$-periodicity then match nicely.

c) The cell average $\bar v_j$ satisfies 
\[
	\cA_{\eps_j}\bar{v}_{j}=\frac{1}{\eps_j}A^{(d)}\partial_d \bar{v}_{j}=0 \qquad \text{in $\T^d$},
\]
instead of only in $\TT^{d-1}\times (0,1)$.
If $v^\sharp_j$ is smooth enough, this can be seen as follows:
Since $v^\sharp_j$ is $\tau_j$-periodic in the first $d-1$ directions, 
\[
	\int_{Q^{d-1}} \partial_{k} v^\sharp_j(\tau_j y',x_d)\dd{y'}=0,
\]
for $k=1,\ldots,d-1$ and every $x_d$. Hence,
\[
\begin{aligned}
	\frac{1}{\eps_j}\partial_d A^{(d)} \bar{v}_{j}(x)
	&= \int_{Q^{d-1}} \frac{1}{\eps_j}A^{(d)}\partial_d v^\sharp_j(\tau_j y',x_d)\dd{y'} \\
	&= \int_{Q^{d-1}} (\cA_{\eps_j} v^\sharp_j)(\tau_j y',x_d)\dd{y'} =0.
\end{aligned}
\]
This means that the relevant component of $\bar{v}_{j}$ (i.e.~its orthogonal projection in $\R^m$ onto $(\ker A^{(d)})^\perp$) is constant in $x_d$, which of course does not change if we extend $\bar{v}_{j}$ $1$-periodically in $x_d$.\\
\end{remark}

\begin{proof}[Proof of Lemma~\ref{lem:localize}]~\\
The space $W^{-1, p}(\T^d;\R^l)$ is the dual of $W^{1,p'}(\T^{d};\R^l)$, which denotes the closure of
the set of $C^{\infty}(\T^d;\R^l)$-functions regarding the Sobolev norm 
$\norm{\frarg}_{W^{1,p'}(Q^d;\R^l)}$.

Since $\cA_{\eps_j}v^\sharp_j=\cA_{\eps_j}\bar{v}_j=0$ in $\T^{d-1}\times (0,1)$,
\[
	\cA_{\eps_j}[\eta(v^\sharp_j-\bar{v}_j)]=\frac{1}{\eps_j} (\partial_d \eta) A^{(d)} (v^\sharp_j-\bar{v}_j) \qquad\text{in $\T^{d-1}\times (0,1)$}.
\]
Therefore, it suffices to show that for every
$\varphi\in W^{1,p'}(\TT^{d};\RR^l)$,
\begin{equation}\label{lemloc1}
	\absB{\int_{Q^{d}} \eta_d'(x_d) A^{(d)}(v^\sharp_j-\bar{v}_j)(x)\cdot \varphi(x)\dd{x}}
	\leq C \tau_j \norm{\varphi}_{W^{1,p'}(\T^d;\R^l)}
\end{equation}
with a constant $C$ independent of $j$ and $\varphi$ (it may depend on $\eta_d$, though). Here, $\eta'_d$ is the derivative of $\eta_d$. 

Actually,
we may even assume that $\varphi\in W_0^{1,p'}(\TT^{d-1}\times (0,1);\RR^l)$, i.e., $\varphi\in W^{1,p'}(\T^d;\R^l)$ with $\varphi=0$ for $x_d\in \{0,1\}$, because the
distance $d_{\eta_d}$ of $\supp \eta_d$ to the boundary of $(0,1)$ is positive. Then the general case of $\varphi\in W^{1,p'}(\TT^{d};\RR^l)$ can be recovered by applying \eqref{lemloc1} to $\tilde{\varphi}(x):=\psi(x_d)\varphi(x)$ for $x\in Q^d$, where $\psi\in C_c^\infty((0,1);[0,1])$ is a fixed function such that $\psi\equiv 1$ on $\supp \eta_d$. This does not change the left-hand side of \eqref{lemloc1}, and $\psi$ can be chosen in such a way that $\norm{\psi'}_{L^\infty(0,1)}\leq C(\eta_d):=\frac{2}{d_{\eta_d}}$, which implies that
\begin{align*}
\norm{\tilde{\varphi}}_{W^{1,p'}(\T^d;\R^l)}\leq (1+\max\{1,C(\eta_d)\})^{\frac{1}{p'}} \norm{\varphi}_{W^{1,p'}(\T^d;\R^l)}.\end{align*}

From now on, let $\varphi\in W_0^{1,p'}(\TT^{d-1}\times (0,1);\RR^l)$. For each $j\in \N$, we set $h_j:=1/\tau_j \in \NN$, and let
$\varphi^\sharp_j$ be the $\tau_j$-periodic function (in $x'$) we get as the average of $h_j^{d-1}$ appropriately shifted copies of $\varphi$, i.e.,
\[
	\varphi^\sharp_j(x',x_d):=\frac{1}{h_j^{d-1}}
	\sum_{i\in \ZZ^{d-1},~i\tau_j\in [0,1)^{d-1}} \varphi(x'+i\tau_j,x_d).
\]
Notice that due to the convexity of the norm, 
\begin{align}\label{est_varphisharp}
	\norm{\varphi^\sharp_j}_{W^{1,p'}(\T^d;\R^l)}\leq \norm{\varphi}_{W^{1,p'}(\T^d;\R^l)}.
\end{align}
By the periodicity of $v_j^\sharp$ and $\varphi_j^\sharp$ in $x'$ and a change of variables, we have that
$$
\begin{aligned}
	&\int_{Q^{d}}\eta_d'(x_d)A^{(d)}(v^\sharp_j-\bar{v}_j)(x)\cdot\varphi(x)\dd{x}\\
	&\qquad=\int_{Q^{d}}\eta_d'(x_d)A^{(d)}(v^\sharp_j-\bar{v}_j)(x)\cdot\varphi^\sharp_j(x)\dd{x}\\
	&\qquad= \int_{Q^{d-1}\times (0,1)}\eta_d'(x_d)A^{(d)}(v^\sharp_j-\bar{v}_j)(\tau_j y',x_d)\cdot\varphi^\sharp_j(\tau_j y',x_d) \dd{(y',x_d)}.
\end{aligned}
$$
In the last line, we can modify $\varphi^\sharp_j$ by any function which is constant in the first $d-1$ variables, because the remaining integrand has average zero in $Q^{d-1}$. In particular, we may use
$\varphi^\sharp_j-\bar{\varphi}_j$ in place of $\varphi^\sharp_j$, where
\[
	\bar{\varphi}_j(x):= 
	\int_{Q^{d-1}} \varphi_j^\sharp(\tau_j y',x_d)\,dy'=\int_{Q^{d-1}} \varphi(y',x_d)\,dy', \qquad x\in Q^d.
\]
Since $\eta'_d$ is bounded, and $(v^\sharp_j)$ (and thus also $(\bar{v}_j)$) is uniformly bounded in $L^p(Q^d;\R^m)$, H\"older's inequality implies that
\begin{align*}
	&\absB{\int_{Q^{d}}\eta_d'(x_d)A^{(d)}(v^\sharp_j-\bar{v}_j)(x)\cdot \varphi(x)\,\dd{x}} \\ &\qquad\qquad \leq C  \Big(\int_{Q^{d-1}\times (0,1)} \abs{\varphi^\sharp_j(\tau_j y',x_d)-\bar{\varphi}_j(\tau_jy', x_d)}^{p'}\dd{(y',x_d)}\Big)^{\frac{1}{p'}}
\end{align*}
with a constant $C$ independent of $j$ and $\varphi$. By Poincar\'e's inequality on $Q^{d-1}$,
\[
	\norm{\varphi^\sharp_j(\tau_j \,\frarg\, ,x_d)-\bar{\varphi}_j(\tau_j\frarg, x_d)}_{L^{p'}(Q^{d-1};\R^l)}
	\leq C\tau_j \norm{(\nabla'\varphi^\sharp_j)(\tau_j \,\frarg\, ,x_d)}_{L^{p'}(Q^{d-1};\R^{l\times (d-1)})}
\]
for almost all $x_d\in (0,1)$,
whence
$$
\begin{aligned}
	&\absB{\int_{Q^{d}}\eta_d'(x_d)A^{(d)}(v^\sharp_j-\bar{v}_j)(x)\cdot\varphi(x)\dd{x}}\\
	&\qquad\leq  C \tau_j \Big(\int_{Q^{d-1}\times (0,1)} \abs{(\nabla' \varphi^\sharp_j)(\tau_j y',x_d)}^{p'} \dd{(y',x_d)}\Big)^{\frac{1}{p'}}\\
	&\qquad= C \tau_j  \norm{\nabla' \varphi^\sharp_j}_{L^{p'}(Q^{d};\R^{l\times (d-1)})}\leq  C  \tau_j\norm{\varphi}_{W^{1,p'}(\T^d;\R^{l})}
\end{aligned}
$$	
by~\eqref{est_varphisharp}, which proves \eqref{lemloc1}.
\end{proof}

In directions tangential to the film, transitions can be created in the usual, straightforward way, without any of the difficulties encountered above.
\begin{lemma}[Localization in \boldmath{$x'$}]\label{lem:localize2}
Let 
$\eps_j\todown 0$, and let $(v_j)\subset L^{p}(Q^{d};\RR^m)$ be a sequence with $v_j\rightharpoonup 0$ in $L^p(Q^d;\R^m)$ and $\cA_{\eps_j} v_j\to 0$ in $W^{-1,p}(\TT^d;\RR^l)$. 
Then, for every function $\varrho:Q^d\to \R$ defined by $\varrho(x)=\varrho'(x')$ for $x\in Q^d$ with $\varrho'\in C_c^\infty(Q^{d-1};[0,1])$,
$$
	\normb{\cA_{\eps_j}[\varrho v_j]}_{W^{-1,p}(\TT^{d};\RR^l)}
	 \To 0 \qquad \text{as $j\to \infty$}.
$$
\end{lemma}
\begin{proof}
We have that
\[
	\cA_{\eps_j}[\varrho v_j]=\Big(\sum_{k=1}^{d-1} (\partial_k \varrho) A^{(k)}v_j
	\Big)+ \varrho \cA_{\eps_j}v_j.
\]
Clearly, $\varrho\cA_{\eps_j}v_j\to 0$ in $W^{-1,p}(\TT^{d};\RR^l)$ by assumption, and by compact embedding, we have that $v_j\rightarrow 0$ strongly in $W^{-1,p}(\TT^{d};\RR^m)$. Consequently, since the functions $\partial_k \varrho$ are uniformly bounded, it follows that
\[
	\sum_{k=1}^{d-1} (\partial_k \varrho) A^{(k)}v_j\To 0  \qquad\text{in $W^{-1,p}(\T^d;\R^l)$ as $j\to \infty$,}
\]
which finishes the proof.
\end{proof}
%
%%%%%%%%%%%%%%%%%%%%%%%%%%%%%%%%%%%%%%%%%%%%%%%%%%%%%%%%%%%%%%%%%%%%%%%%%%%%%%%%%%%%%%%%%%%%%%%%%%%%%%%%%%%%%%%%%%%%%%%%%%%%%%%%%%%%%%%%
\subsection{Upper bound III: The general case}
The construction of a recovery sequence in the general case yields the following general upper bound.
\begin{proposition}\label{prop:upperboundA}
Let $\alpha>1$ and $\eps_j\todown 0$, and assume that (H0), (H2)--(H5) and (A1)--(A4) hold. 
Then, for every $u\in \Ucal_{\Acal_0}$, there is a sequence $(u_j)$ with 
$u_j\in \Ucal_{\Acal_{\eps_j}}$ $(j\in \N)$ 
such that $u_j\weakly u$ in $L^p(\Omega_1;\R^m)$ for $j\to \infty$ and
\begin{align*}
\limsup_{j\to \infty} \int_{\Omega_1}f\Bigl(\frac{x'}{\eps_j^\alpha}, \frac{x_d}{\eps_j^{\alpha-1}}, u_j(x)\Bigr)\dd{x} \leq \int_{\Omega_1} 
f^{\rm hom}_{\Acal}(u(x))\dd{x},
\end{align*}
with $f_{\Acal}^{\rm hom}:\RR^m\to \RR$ defined by \eqref{cell}.
\end{proposition}
The strategy of the proof can be summarized as follows: We first approximate $u$ with a piecewise constant field. On each of the small cubes on which the approximation is constant, we obtain an $\cA_{\eps_j}$-free recovery sequence according to Proposition~\ref{prop:ubA-constant}. 
Then, by using suitable cut-off functions, we glue all these sequences together to obtain one on $\TT^d$. In general, this combined sequence will no longer be $\cA_{\eps_j}$-free, but Lemma~\ref{lem:localize} and Lemma~\ref{lem:localize2} make sure that the effect of the cut-off in terms of violation of the PDE constraint remains small enough. In general, both $\alpha>1$ and the convexity of $f$ play a crucial role: The sequence from Proposition~\ref{prop:ubA-constant} can be assumed to be $\tau_j:=\eps_j^\alpha$-periodic in the first $d-1$ variables as required in Lemma~\ref{lem:localize}, because otherwise, we can replace it by an average of shifted copies, which does not increase the energy due to convexity.
Finally, we project back onto $\cA_{\eps_j}$-fields. Since the projection error converges to zero strongly in $L^p$ as $j\to \infty$, uniform continuity properties then yield the assertion. 
\begin{proof}
It is enough to show that there exists $(w_j)\subset L^p(\Omega_1;\RR^m)$ such that $\cA_{\eps_j}w_j=0$ in $\Omega_1$, $w_j\rightharpoonup 0$ in $L^p(\Omega_1;\R^m)$ and
\begin{align}\label{p:ub-1b}
	\limsup_{j\to \infty} \int_{\Omega_1}f\Bigl(\frac{x'}{\eps_j^\alpha}, \frac{x_d}{\eps_j^{\alpha-1}}, u+w_j\Bigr)\dd{x} \leq \int_{\Omega_1} 
f^{\rm hom}_{\Acal}(u)\dd{x},
\end{align}
because $u\in \Ucal_{\Acal_0}$ can be approximated strongly in $L^p(\Omega_1;\R^m)$ by a sequence in $\Ucal_{\Acal_{\eps_j}}$ due to Lemma~\ref{lem:A0}.
Here and throughout the proof, it is good to keep in mind that the operator  $L^p(\Omega_1;\R^m)\to L^1(\Omega_1)$, $w\mapsto f\left(\frac{\cdot}{\lambda},w(\cdot)\right)$, is uniformly continuous on bounded subsets of $L^p(\Omega_1;\R^m)$, also uniformly in $\lambda\in (0,1]$, as already pointed out in Remark~\ref{rem:ucont}. In particular, this means that in \eqref{p:ub-1b}, any sequence converging strongly to zero in $L^p$ can be added to $u$ or $(w_j)$ without changing the $\limsup$.

\vspace*{1ex}
\noindent\emph{Step 1: Approximation of $u$ with a piecewise constant field.}\\
Let $\gamma>0$, and define $u=0$ on $Q^d\setminus \Omega_1$.
By standard approximation results in $L^p$, there exist $h>0$ (small enough, with $1/h\in \NN$) and values $\xi^{(k)}\in \R^m$, $k=1, \ldots, h^{-d}$, such that on the uniform 
cubical grid with cell size $h$, given by $h^{-d}$ cubes $Q_k$ which are shifted copies of $h Q^d$ forming a pairwise disjoint convering of $Q^d$ (up to a set of measure zero), 
\[
	\norm{u^{(h)}-u}_{L^p(Q^d;\RR^m)}\leq \gamma
\]
for the piecewise constant function
\[
	u^{(h)}:=\sum_{k=1}^{h^{-d}} \chi_{Q_k} \xi^{(k)}.
\]
Notice that $u^{(h)}$ is not necessarily $\cA_0$-free.
Below, we also use the notation
\[
	\Omega_1^{(h)}:=\bigcup_{k\in K(h)} Q_k,\quad\text{where}\quad
	K(h):=\{k\in \N\mid Q_k\subset \Omega_1\}.
\]
Since $\absb{\Omega_1\setminus \Omega_1^{(h)}}\to 0$ as $h\todown 0$, we can assume that $\xi^{(k)}=0$ for every $k\notin K(h)$, i.e., $u^{(h)}=0$ on $Q^d\setminus \Omega_1^{(h)}$.

\vspace*{1ex}
\noindent\emph{Step 2: Piecewise constant $u$, assuming in addition that $\eps_j^{-\alpha}\in \NN$ for each $j\in \N$.}\\
Let the mesh size $h$, the associated cubes $Q_k$ and the values $\xi^{(k)}$ be fixed.
We claim that for every $\delta>0$, there exists $(w_j)\subset L^p(Q^d;\RR^m)$ such that 
$w_j\rightharpoonup 0$ in $L^p(Q^d;\R^m)$, $w_j=0$ on $Q^d\setminus \Omega_1^{(h)}$, 
$\cA_{\eps_j}w_j=0$ in $\TT^d$,
and 
\begin{align}\label{p:ub-2}
	\limsup_{j\to \infty} \int_{\Omega_1^{(h)}}f\Bigl(\frac{x'}{\eps_j^\alpha}, \frac{x_d}{\eps_j^{\alpha-1}}, u^{(h)} + w_j\Bigr)\dd{x} \leq 
	\int_{\Omega_1^{(h)}} 
f^{\rm hom}_{\Acal}(u^{(h)})\dd{x}+\delta.
\end{align}
Since $u^{(h)}=w_j=0$ on $\Omega_1\setminus \Omega_1^{(h)}$, \eqref{p:ub-2} and (H3) imply that 
\begin{align}\label{p:ub-2b}
	\limsup_{j\to \infty} \int_{\Omega_1}f\Bigl(\frac{x'}{\eps_j^\alpha}, \frac{x_d}{\eps_j^{\alpha-1}}, u^{(h)}+w_j\Bigr)\dd{x} \leq 
	\int_{\Omega_1} 
f^{\rm hom}_{\Acal}(u^{(h)})\dd{x}+\delta+c_1\absb{\Omega_1\setminus \Omega_1^{(h)}}.
\end{align}
For each $k\in K(h)$, Proposition~\ref{prop:ubA-constant} (applied with $\Omega_1=Q_k$) yields a sequence
$(v_j^{(k)})_j\subset L^p(Q^d;\RR^m)$ such that $\cA_{\eps_j}v_j^{(k)}=0$ in $\TT^{d-1}\times (0,1)$ (cf.~Remark~\ref{rem:ubA-constant}), $v_j^{(k)}\rightharpoonup 0$ in $L^p(Q^d;\R^m)$ as $j\to \infty$, and
\begin{align}\label{p:ub-3}
	\limsup_{j\to \infty} \int_{Q_k}f\Bigl(\frac{x'}{\eps_j^\alpha}, \frac{x_d}{\eps_j^{\alpha-1}}, \xi^{(k)} + v_j^{(k)}\Bigr)\dd{x} \leq 
	\int_{Q_k} 
f^{\rm hom}_{\Acal}(\xi^{(k)})\dd{x}.
\end{align}
Define 
\[
	(v_j^{(k)})^{\sharp}(x',x_d):=\eps_j^{(d-1)\alpha}\sum_{l'\in (\eps_j^\alpha \ZZ^{d-1})\cap [0,1)^{d-1}}
	v_j^{(k)}(x'+l',x_d),
\]
which is an $\eps_j^\alpha$-periodic function in each component of $x'$ (since $\eps_j^{-\alpha}\in \NN$), also satisfying $\cA_{\eps_j}(v_j^{(k)})^\sharp=0$ in $\TT^{d-1}\times (0,1)$ and $(v_j^{(k)})^\sharp \rightharpoonup 0$ in $L^p(Q^d;\R^m)$ as $j\to \infty$. By the convexity and periodicity of $f$ and a change of variables,
\begin{align*}
	&\int_{Q_k}f\Bigl(\frac{x'}{\eps_j^\alpha}, \frac{x_d}{\eps_j^{\alpha-1}}, \xi^{(k)} + (v_j^{(k)})^\sharp(x)\Bigr)\dd{x} \\
	&\qquad\leq \eps_j^{(d-1)\alpha}\sum_{l'\in (\eps_j^\alpha \ZZ^{d-1})\cap [0,1)^{d-1}}  
	\int_{Q_k}f\Bigl(\frac{x'}{\eps_j^\alpha}, \frac{x_d}{\eps_j^{\alpha-1}}, \xi^{(k)}+v_j^{(k)}(x'+l',x_d)\Bigr)\dd{x} \\
	&\qquad=\int_{Q_k}f\Bigl(\frac{x'}{\eps_j^\alpha}, \frac{x_d}{\eps_j^{\alpha-1}}, \xi^{(k)} + v_j^{(k)}(x)\Bigr)\dd{x}.
\end{align*}
Therefore, we may assume without loss of generality that $v_j^{(k)}=(v_j^{(k)})^\sharp$ in \eqref{p:ub-3}, i.e., that $v_j^{(k)}$ is 
$\eps_j^\alpha$-periodic in its first $d-1$ variables. Also notice that $((v_j^{(k)})^\sharp)_j$ inherits the $p$-equintegrability of
$(v_j^{(k)})_j$.

We will now cut off $v_j^{(k)}$ near $\partial Q_k$, controlling the effect on $\cA_{\eps_j}v_j^{(k)}$ with the localization lemmas of the previous subsection. Recall that the cell averages introduced in Lemma~\ref{lem:localize},
\[
	\bar{v}_j^{(k)}(x):=\int_{Q^{d-1}} v_j^{(k)}(\eps_j^\alpha x',x_d)\dd{x'},\qquad x\in Q^d,
\]
satisfy $\cA_{\eps_j} \bar{v}_j^{(k)}=\frac{1}{\eps_j}A^{(d)}\partial_d \bar{v}_j^{(k)}=0$ on $\TT^{d-1}\times (0,1)$ as shown in Remark~\ref{rem:proplocal}. 
As a consequence, with $P^{(d)}:\RR^m\to \RR^m$ denoting the orthogonal projection of 
$\R^m$ onto $\ker A^{(d)}$ (so that $A^{(d)}P^{(d)}=0$ and $A^{(d)}:(I-P^{(d)})(\RR^m)\to A^{(d)}(\RR^m)$ is invertible),
\[
	\partial_d (I-P^{(d)})\bar{v}_j^{(k)}=0\quad\text{on $Q^d$},
\]
so $(I-P^{(d)})\bar{v}_j^{(k)}$ is constant for every $j\in \N$.
As $\bar{v}_j^{(k)}\rightharpoonup 0$ in $L^p(Q^d;\R^m)$ for $j\to \infty$ just like $(v_j^{(k)})_j$, one therefore obtains that
\begin{align}\label{projection_convergence}
	\norm{(I-P^{(d)})\bar{v}_j^{(k)}}_{L^p(Q^d;\RR^m)} {\To} \,0 \qquad \text{as $j\to \infty$}.
\end{align}
For each $k$, choose functions $\eta^{(k)}$ and $\varrho^{(k)}$ as in Lemma~\ref{lem:localize} and Lemma~\ref{lem:localize2}, respectively, 
such that their product $\eta^{(k)}\varrho^{(k)}$ is supported in $Q_k$, and the set where $\eta^{(k)}\varrho^{(k)}\neq 1$ in 
$Q_k$ is small enough so that for all sufficiently large $j$,
\begin{align}\label{p:ub-4}
	\abs{
	\int_{Q_k}f\Bigl(\frac{x'}{\eps_j^\alpha}, \frac{x_d}{\eps_j^{\alpha-1}}, \xi^{(k)}+v_j^{(k)}(x)\Bigr)\dd{x}
	-
	\int_{Q_k}f\Bigl(\frac{x'}{\eps_j^\alpha}, \frac{x_d}{\eps_j^{\alpha-1}}, \xi^{(k)} + \tilde{w}_j(x)\Bigr)\dd{x}
	}\leq h^d\delta,
\end{align}
where
\[
	\tilde{w}_j:=\sum_{k\in K(h)}
	\varrho^{(k)}\eta^{(k)}\left[\big(v_j^{(k)}-\bar{v}_j^{(k)}\big)+P^{(d)} \bar{v}_j^{(k)}\right] \in L^p(Q^d;\R^m).
\]
%with $\supp \tilde{w}_j\subset \Omega_1^{(h)}$.
Here, we use that $(v_j^{(k)})_j$ is $p$-equiintegrable and bounded in $L^p(Q^d;\R^m)$ for each $k$, as well as~\eqref{projection_convergence}, and the uniform continuity of $f$ as an operator from $L^p$ to $L^1$ restricted to bounded subsets. 
Due to $A^{(d)}P^{(d)}=0$, we also have that
\[
	\cA_{\eps_j} \Big[\sum_{k\in K(h)} \eta^{(k)} P^{(d)} \bar{v}_j^{(k)} \Big]
	=\frac{1}{\eps_j}\partial_d \sum_{k\in K(h)} \eta^{(k)} A^{(d)}P^{(d)} \bar{v}_j^{(k)}=0
	\quad\text{in $\TT^d$.}
\]
Therefore, Lemma~\ref{lem:localize} (with $\tau_j:=\eps_j^\alpha$, exploiting that $\alpha>1$) and Lemma~\ref{lem:localize2} are successively applicable for each $k$, and we get that
\[
	\norm{ \cA_{\eps_j} \tilde{w}_j}_{W^{-1,p}(\TT^{d};\RR^l)}{\To}0\qquad \text{as $j\to \infty$}.
\]
The projection 
\[
	w_j:=\Pcal_{\cA_{\eps_j}}\tilde{w}_j, \qquad j\in \N,
\]
according to Lemma~\ref{theorem_projection}
thus yields that $w_j\in L^p(Q^d;\R^m)$ with $\cA_{\eps_j} w_j=0$ in $\TT^d$ and $w_j-\tilde{w}_j\to 0$ in $L^p(Q^d;\R^m)$, so $w_j\weakly 0$ in $L^p(Q^d;\R^m)$ by~\eqref{projection_convergence} and the weak convergence of  $(\bar{v}_j^{(k)})_j$, $(v_j^{(k)})_j$ to zero. 
In particular, $(w_j)$ inherits the $p$-equiintegrability of $(\tilde{w}_j)$.
Finally, 
in view of the definitions of $u^{(h)}$ and $w_j$, \eqref{p:ub-4}, and~\eqref{p:ub-3}, it follows that
\begin{align*}
	&\limsup_{j\to \infty} \int_{\Omega_1^{(h)}}f\Bigl(\frac{x'}{\eps_j^\alpha}, \frac{x_d}{\eps_j^{\alpha-1}}, u^{(h)} + w_j\Bigr)\dd{x} \\
	&\qquad=\limsup_{j\to \infty} \sum_{k\in K(h)}
	\int_{Q_k}f\Bigl(\frac{x'}{\eps_j^\alpha}, \frac{x_d}{\eps_j^{\alpha-1}}, \xi^{(k)} + \tilde{w}_j(x) \Bigr)\dd{x}\\
	&\qquad \leq\sum_{k\in K(h)} \limsup_{j\to \infty}
	\int_{Q_k}f\Bigl(\frac{x'}{\eps_j^\alpha}, \frac{x_d}{\eps_j^{\alpha-1}}, \xi^{(k)} + v_j^{(k)} \Bigr)\dd{x} + \delta h^d \# K(h)\\
	&\qquad \leq \sum_{k\in K(h)} 
	\int_{Q_k} 
f^{\rm hom}_{\Acal}(\xi^{(k)})\dd{x}+\delta = \int_{\Omega_1^{(h)}} f^{\rm hom}_{\Acal_0}(u^{(h)})\dd{x}  + \delta,
\end{align*}
which proves~\eqref{p:ub-2}.

\vspace*{1ex}
\noindent\emph{Step 3: Piecewise constant $u$, in case $\eps_j^{-\alpha}\notin \NN$.}\\
Let $\delta>0$ and $(\tilde{w}_j)$ denote the sequence of Step 2, applied with $\tilde{\eps}_j:=\theta_j \eps_j$ for $j\in \N$, where
\[
	\theta_j:=\Big(\frac{\eps_j^{-\alpha}}{\lceil \eps_j^{-\alpha} \rceil}\Big)^{\frac{1}{\alpha}}\leq 1.
\]	
By construction, $\theta_j\to 1$ as $j\to \infty$, $\tilde{\eps}_j^{-\alpha}=\lceil \eps_j^{-\alpha} \rceil\in \NN$. We set
\[
	w_j(x):=\tilde{w}_j(\theta_j^\alpha x',\theta_j^{\alpha-1} x_d), \qquad x\in Q^d, j\in \N,
\]
which defines a $p$-equiintegrable sequence in $\Ucal_{\Acal_{\eps_j}}$, since $\cA_{\tilde{\eps}_j}\tilde{w}_j=0$ in $\TT^d$. 
A change of variables shows that
\begin{align*}
&\limsup_{j\to \infty} \int_{\Omega_1}f\Bigl(\frac{x'}{\eps_j^\alpha}, \frac{x_d}{\eps_j^{\alpha-1}}, u^{(h)}(x)+ w_j(x)\Bigr)\dd{x} \\
&\qquad=
\limsup_{j\to \infty} \int_{\theta_j^\alpha \omega\times (0,\theta_j^{\alpha-1})}f\Bigl(\frac{x'}{\tilde{\eps}_j^\alpha}, \frac{x_d}{\tilde{\eps}_j^{\alpha-1}},
\tilde{w}_j(x)+u^{(h)}(\theta_j^{-\alpha}x',\theta_j^{1-\alpha} x_d)\Bigr)\dd{x}.
\end{align*}
The integrand on the right-hand side is equiintegrable due to (H2) and the $p$-equiintegrability of $(w_j)$. 
Since $\theta_j^\alpha \omega\times (0,\theta_j^{\alpha-1})$ and $\Omega_1=\omega\times (0,1)$ only differ on a set whose measure converges to zero, this means we can change the domain of integration back to $\Omega_1$, using the convention that $u^{(h)}=0$ on $\RR^d\setminus \Omega_1$.
In addition, $u^{(h)}(\theta_j^{-\alpha}x',\theta_j^{1-\alpha}x_d)\to u^{(h)}(x',x_d)$ in $L^p(\RR^d;\RR^m)$. 
In view of \eqref{p:ub-2b}, we therefore conclude that
\begin{align*}
	&\limsup_{j\to \infty} 
		\int_{\Omega_1}f\Bigl(\frac{x'}{\eps_j^\alpha}, \frac{x_d}{\eps_j^{\alpha-1}}, u^{(h)}(x) + w_j(x)\Bigr)\dd{x}  \\
	&\qquad=\limsup_{j\to \infty} 
		\int_{\Omega_1}f\Bigl(\frac{x'}{\tilde{\eps}_j^\alpha}, \frac{x_d}{\tilde{\eps}_j^{\alpha-1}}, 
	u^{(h)}(\theta_j^{-\alpha}x',\theta_j^{1-\alpha}x_d)+\tilde{w}_j(x)\Bigr)\dd{x} \\
	&\qquad=\limsup_{j\to \infty} 
		\int_{\Omega_1}f\Bigl(\frac{x'}{\tilde{\eps}_j^\alpha}, \frac{x_d}{\tilde{\eps}_j^{\alpha-1}}, 
		u^{(h)}(x)+\tilde{w}_j(x)\Bigr)\dd{x}\\
	&\qquad\leq \int_{\Omega_1} f^{\rm hom}_{\Acal}(u^{(h)})\dd{x}+\delta+c_1\absb{\Omega_1\setminus \Omega_1^{(h)}}.
\end{align*}

\noindent\emph{Step 4: Conclusion by diagonalization.}\\
According to Step~1, there is a sequence of mesh sizes $h_i\todown 0$ for $i\to\infty$ such that
$\norm{u-u^{(h_i)}}_{L^p(Q^d;\R^m)}\leq \frac{1}{i}$. Therefore,
\[	
	\int_{\Omega_1}f^{\rm hom}_{\Acal}\big(u^{(h_i)}\big)\dd{x}
	-\int_{\Omega_1}f^{\rm hom}_{\Acal} \big(u\big)\dd{x}\ {\To} \ 0\qquad \text{as $i\to \infty$},
\]
and, with the sequence $(w_j)=(w_j^{(h_i)})_j$ from Step~2 (or Step~3) with $\delta:=\frac{1}{i}$,
\[	
	\int_{\Omega_1}f\Bigl(\frac{x'}{\eps_j^\alpha}, \frac{x_d}{\eps_j^{\alpha-1}}, u^{(h_i)}+w_j^{(h_i)}\Bigr)\dd{x}
	-\int_{\Omega_1}f\Bigl(\frac{x'}{\eps_j^\alpha}, \frac{x_d}{\eps_j^{\alpha-1}}, u + w_j^{(h_i)}\Bigr)\dd{x}\ {\To}0
\]
as $i\to \infty$, uniformly in $j$. By Step 2 (or Step 3), we also have for all $i\in \N$ that
\begin{align*}
	\limsup_{j\to \infty} 
		\int_{\Omega_1}f\Bigl(\frac{x'}{\eps_j^\alpha}, \frac{x_d}{\eps_j^{\alpha-1}}, u^{(h_i)} + w_j^{(h_i)}\Bigr)\dd{x}  
	&\leq \int_{\Omega_1} f^{\rm hom}_{\Acal}(u^{(h_i)})\dd{x}+\frac{1}{i}+c_1\absb{\Omega_1\setminus \Omega_1^{(h_i)}}.
\end{align*}
Clearly, $\frac{1}{i}+c_1\absb{\Omega_1\setminus \Omega_1^{(h_i)}}\to 0$ as $i\to \infty$, and
by coercivity (H4), $(w_j^{(h_i)})_{j,i}$ is bounded in $L^p$ (uniformly in $j$ and $i$). 
Therefore, we can select an appropriate diagonal subsequence $(w_j^{(h_{i_j})})_j$ which yields the assertion.
\end{proof}

%%%%%%%%%%%%%%%%%%%%%%%%%%%%%% Simultaneous limits II: Thin films with coarse heterogeneity ($\alpha\leq 1$) %%%%%%%%%%%%%%%%%%%%%%%%%%%%%%%%%%%%%%%%
\section{Thin films with coarse heterogeneity ($\alpha\leq 1$)}\label{sec:alpha<1}

As pointed out in Section~\ref{sec:regimes&results}, the regime $\alpha\leq 1$ is qualitatively different from the case $\alpha>1$, where we gave a local characterization of the simultaneous $\Gamma$-limit of $(F_{\eps, \eps^\alpha})$.
Here, we will provide, for any constant-rank operator $\Acal$ and every $\alpha\leq 1$, an explicit example of an energy density $f$ with properties (H0)-(H5) such that $\Gamma$-$\liminf_{\eps\to 0} F_{\eps, \eps^\alpha}$ is nonlocal.
The key idea behind the construction is finding macroscopically compatible phases that consist of oscillating, microscopically incompatible phases, and thus create additional line energy when joined together. 

\begin{remark}
Let $\alpha=1$. We assume (A1)-(A4), H(0), (H2)-(H5), and in addition that $f$ is independent of $z'$ (compare Remark~\ref{lem:properties_fhom} b)). 
Then it is a straightforward application of~\cite[Theorem~1.1]{KR14} to derive that
\begin{align*}
\Gamma\text{-}\lim_{\eps\to 0} F_{\eps, \eps}(u)= 
\begin{cases}
\int_{\Omega_1} f(x_d, u(x))\dd{x} & \text{if $\Acal_0 u=0$ in $\Omega_1$,}\\
\infty & \text{otherwise,}
\end{cases}
\qquad u\in L^p(\Omega_1;\R^m).
\end{align*}
Of course, this limit functional is local. 
If $f$ is continuous in $z_d$, an analogous observation
can be made for $\alpha<1$.
Hence, for  a counterexample to local behavior we need to exploit decisively the freedom of choosing $f$ to be $z'$-dependent.
\end{remark}

\subsection{Construction tools}
Before we can formulate the counterexample the following preparations are needed.
\begin{lemma}\label{lem:general_counterexample}\label{lem:generalized_vectorquartett}
Let $\Acal$ be a constant-rank operator with $\rank \Abb(\eta)=r$ for all $\eta\in \R^d\setminus\{0\}$ that satisfies (A2).
If $m>r$ and $n\in \Scal^{d-1}:=\{\eta\in \R^d:\abs{\eta}=1\}$ such that $\ker \Abb(n)\neq \ker \Abb(e_d)$, then there are vectors 
$\xi_1, \xi_2, \sigma_1, \sigma_2\in \R^m$ such that
\begin{align}
(\sigma_1+\sigma_2) -(\xi_1+\xi_2) &\in \ker \Abb_0(e_d)\setminus\{0\}, \label{lem:counterexample1}\\
\xi_1-\xi_2 &\in \ker\Abb(n),\label{lem:counterexample2}\\
\sigma_1-\sigma_2 &\in \ker\Abb(n),\nonumber\\
\xi_1-\sigma_1\notin \ker \Abb_0(e_d)=\ker \Abb(e_d) &\text{ or }\xi_2-\sigma_2\notin \ker \Abb_0(e_d)=\ker\Abb(e_d),\label{lem:counterexample4}
\end{align} 
and 
\begin{align}\label{lem:counterexample5}
\xi_1-\sigma_1, \text{ $\xi_2-\sigma_2$ are linearly independent.}
\end{align}
\end{lemma}

\begin{remark}\label{rem:choice_n}
Requiring the existence of $n\in \Scal^{d-1}$ such that $\ker \Abb(n)\neq \ker \Abb(e_d)$ is equivalent to postulating that 
not all kernels of the matrices $A^{(k)}\in \R^{l\times m}$ coincide, or to saying that $\ker \Abb(\eta)$ is not constant in $\eta\neq 0$.
In fact, the constant-rank property of $\Acal$ entails the equivalence of the following two statements:
\begin{align*}
\ker \Abb(n)=\ker \Abb(e_d) \quad \text{for all $n\in \Scal^{d-1}$,}\end{align*}
if and only if
\begin{align*}
\ker A^{(k)}=\ker A^{(d)}\quad \text{for all $k=1, \ldots, d-1$.}
\end{align*}
In particular, without loss of generality $n$ can be chosen to be a standard unit vector $e_i\neq e_d$
for some $i=1, \ldots, d-1$. 
\end{remark}

\begin{proof}
Recall that $r\leq \min\{m,l\}$. 
In view of $m>r$, together with \cite[Section~2.4]{KR14}, we find that 
$\ker \Abb_0(e_d)=\ker \Abb(e_d)=\ker A^{(d)} \neq \{0\}$. So, let $z\in \ker \Abb_0(e_d)\setminus\{0\}\subset \R^m$.

On the other hand, we fix a
\begin{align*}
v\in \ker \Abb(n) \setminus \ker \Abb(e_d).
\end{align*}
The existence of such a $v$ follows from the assumption that $\ker \Abb(n)\neq \ker \Abb(e_d)$, together with 
the fact that $\dim \ker \Abb(n)=\dim \ker \Abb(e_d)$ by the constant-rank property of $\Acal$ and the rank-nullity theorem.

Let $\alpha, \beta\in \R$ with $\alpha\neq 0$. 
Further, we let $\sigma_1, \sigma_2\in \R^m$ be such that $\sigma_1-\sigma_2=\beta v$, and set
\begin{align*}
\xi_1 =\sigma_1+ \alpha v+z\qquad \text{and}\qquad \xi_2=\sigma_2 -\alpha v +z.
\end{align*}

Then, a straightforward computation shows that these quantities satisfy all the required properties.
Indeed,
\begin{align*}
\xi_1-\sigma_1=\alpha v+z &\notin \ker \Abb_0(e_d),\\  
\xi_2-\sigma_2=-\alpha v+z &\notin\ker  \Abb_0(e_d),
\end{align*}
are linearly independent, since $v$ and $z$ are, and $\alpha\neq 0$. Besides,
\begin{align*}
(\sigma_1+\sigma_2) -(\xi_1+\xi_2) =-2z  &\in \ker \Abb_0(e_d)\setminus\{0\},\\
\sigma_1-\sigma_2 = \beta v & \in \ker \Abb(n),\\
\xi_1-\xi_2= (\beta+ 2\alpha) v & \in \ker \Abb(n).
\end{align*}
\end{proof}

\begin{example}\label{ex:nonlocaldivcurl}
a) In the case $\Acal=\diverg$, we have $m=d>1$, $l=1$ and $r=1$. Observing that 
$\ker \Abb_{\diverg}(\xi)=\{v\in \R^d:\xi\cdot v=0\}$ for all $\xi\in \R^d\setminus \{0\}$ guarantees that
\begin{align*}\ker \Abb_{\diverg}(e_i)\neq \ker \Abb_{\diverg}(e_d).
\end{align*} for all $i=1, \ldots, d-1$.

By choosing $n=e_1$, $z=-e_1$, $v=e_d$, $\beta=-2$ and $\alpha=3$ one finds in particular that 
the vectors $\xi_1=2 e_d$, $\xi_2=-2e_d$, $\sigma_1=e_1-e_d$, and $\sigma_2=e_1+e_d$
feature the properties of Lemma~\ref{lem:generalized_vectorquartett}.

b) For $\Acal=\curl$ applied to matrix-valued functions $\R^d\to \R^{n\times d}\cong \R^m$ with $d>1$, 
one has $m=nd$, $l=n\frac{d(d-1)}{2}$ and $r=n(d-1)$. From
\begin{align*}
\ker \Abb_{\curl} (\xi)=\{a\otimes \xi: a\in \R^n\}\subset \R^{n\times d}
\end{align*}
along with the choice of the unit vector $n=e_1$ and the quantities $z=e_1\otimes e_d$, $v=e_1\otimes e_1$, $\alpha=1$ and $\beta=2$, 
we infer that $\xi_1=3e_1\otimes e_1+ 2 e_1\otimes e_d$, $\xi_2=-e_1\otimes e_1 + 2e_1\otimes e_d$, $\sigma_1=2e_1\otimes e_1 +e_1\otimes e_d$ and $\sigma_2=e_1\otimes e_d$
satisfy Lemma~\ref{lem:generalized_vectorquartett}.
\end{example}

%%%%%%%%%%%%%%%%%%%%%%%%%%%%%%%%%%%%%%%%%%%%%%%%%%%%%%%%%%%%%%%%%%%
\subsection{Counterexample to locality}
In the following, we assume for simplicity that $\Omega_1=Q^d$.
With the quantities of Lemma~\ref{lem:general_counterexample}, supposing that $n\neq e_d$ is a standard unit vector in $\R^d$ (compare Remark~\ref{rem:choice_n}) and $p>1$, we define the function $g:
\R^{d-1}\times \R^m\to \R$ by
\begin{align*}
g(x', v)=\begin{cases}
\abs{v-\xi_1}^{p/2}\abs{v-\sigma_1}^{p/2} & \text{if $x'\cdot n'\leq \frac{1}{2}$,}\\
\abs{v-\xi_2}^{p/2}\abs{v-\sigma_2}^{p/2} & \text{if $x'\cdot n'>\frac{1}{2}$}, 
\end{cases}
\qquad x'\in Q^{d-1}, v\in \R^m,
\end{align*} 
and periodic extension in the first variable.
Let us denote the convexification of $g$ with respect to the second variable by $f$, so 
\begin{align}\label{def:f}
f(x', \frarg)=[g(x', \frarg)]^{\rm c}=:g^c(x', \frarg)
\end{align}
for all $x'\in \R^{d-1}$. Notice the growth and coercivity properties 
\begin{align}\label{coercivity_g}
c\abs{v}^p-C \leq f(x', v)\leq C(1+\abs{v}^p), \qquad x'\in \R^{d-1}, v\in \R^m,
\end{align}
with constants $c, C>0$.

For $\eps>0$ and an open set $D\subset\Omega_1=Q^d$, we define the functional 
$F_\eps^\sharp(\frarg;D):L^p(\Omega_1; \R^m)\to [0, \infty]$ by
\begin{align*}
F_\eps^\sharp(u;D):=\begin{cases}
\int_D f\bigl(\frac{x'}{\eps^\alpha}, u(x)\bigr)\dd{x} & \text{if $\Acal_\eps u =0$ in $\T^{d-1}\times (0,1)$,}\\
\infty & \text{else,}
\end{cases}\qquad u\in L^p(\Omega_1;\R^m).
\end{align*}
For a sequence $\eps_j\todown 0$, the $\Gamma$-$\liminf$ of $(F_{\eps_j}^\sharp)$ is denoted by
\begin{align*}
F_0^\sharp(u;D)=\inf \{\liminf_{j\to \infty} F_{\eps_j}^\sharp(u_j; D) : (u_j)\subset L^p(\Omega_1;\R^m), u_j\weakly u \text{ in $L^p(\Omega_1; \R^m)$}\}.
\end{align*}
The following result implies Theorem~\ref{theo:simultaneous_nonlocal}, taking Lemma~\ref{lem:localize2} on the localization in $x'$ into account.

\begin{proposition}[Nonlocal character of \boldmath{$F_0^\sharp$} for \boldmath{$\alpha\leq 1$}]\label{prop:nonlocal_alpha<1}
Let $\Acal$ and $n$ be as in Lemma~\ref{lem:generalized_vectorquartett}, assuming that $n\neq e_d$ is a 
standard unit vector in $\R^d$, and consider a given sequence {$\eps_j\todown 0$}. Further, let $f:\R^{d-1}\times \R^m \to \R$ be the function defined in \eqref{def:f}, 
and $\Omega_1=Q^{d}$.
For $x\in \Omega_1$ let
\begin{align*}
u_0(x)=\begin{cases} \bar{\sigma}& \text{if $x_d>1/2$,}\\
\bar{\xi} & \text{if $x_d\leq 1/2$},
\end{cases}
\end{align*} 
where $\bar{\xi}=\frac{1}{2}(\xi_1+\xi_2)$ and $\bar{\eta}=\frac{1}{2}(\sigma_1+\sigma_2)$.

Then,  
\begin{align*}
F_0^\sharp(u_0; \omega \times (0,1/2)) = 0 = F_0^\sharp(u_0; \omega \times (1/2, 1)).
\end{align*}

If $\alpha\leq 1$ and $\eps_j^\alpha=\frac{1}{j}$ for all $j\in \N$, then
$F_0^\sharp(u_0; \Omega_1)>0.$
\end{proposition}
\begin{remark}
More general domains ($\Omega_1\neq Q^{d}$) and general sequences ($\eps_j^\alpha\neq \frac{1}{j}$) are possible, at the expense of a few additional technicalities in the proof.
\end{remark}
\begin{proof} Notice that by~\eqref{lem:counterexample1} and Lemma~\ref{lem:Afree_jumps} one has
$\Acal_0 u_0=0$ in $\T^{d-1}\times (0,1)$.
Without loss of generality we assume in the following that $n=e_1$. 
We split the proof in three steps. 

\textit{Step~1: $F_0^\sharp(u_0; \omega\times (0,1/2))=0$.} 
Consider the $Q^d$-periodic function 
\begin{align*}
v(y)= \begin{cases} 
\xi_1 & \text{ if $y_1=y\cdot n\leq 1/2$,}\\ 
\xi_2 & \text{ if $y_1=y\cdot n > 1/2$,}
\end{cases} \qquad y\in Q^d,
\end{align*}
and define $v_{j}(x)=v(\frac{x}{\eps_j^\alpha})$ for $x\in \Omega_1$. 
Then, 
$v_{j}\weakly \dashint_{Q^d} v(y)\dd{y} = \bar{\xi}$ in $L^p(\Omega_1;\R^m)$ and 
\begin{align*}
\Acal_{\eps_j} v_{j} = 0\qquad \text{in $\T^{d-1}\times (0,1)$}
\end{align*} 
in view of 
$\xi_1-\xi_2\in \ker \Abb(n)=\ker\Abb_{\eps_j}(e_1)$ by~\eqref{lem:counterexample2} and Lemma~\ref{lem:Afree_jumps}.
In addition, due to the definition of $v_{j}$ one has for all $x\in \Omega_1$ that
\begin{align*}
f\Bigl(\frac{x'}{\eps^\alpha_j}, v_j(x)\Bigr)=g\Bigl(\frac{x'}{\eps^\alpha_j}, v_{j}(x)\Bigr)
=0.
\end{align*}
Thus, 
\begin{align*}
F_0^\sharp(u_0; \omega \times (0,1/2)) = F_0^\sharp(\bar{\xi}; \omega\times (0,1/2)) 
\leq \liminf_{j\to \infty} F_{\eps_j}^\sharp(v_{j};\omega\times(0,1/2))=0,
\end{align*}
which finishes the proof of this step.

\textit{Step~2: $F_0^\sharp(u_0; \omega\times (1/2,1))=0$.} 
The proof is in complete analogy to Step~1, just replace 
$\xi_1$ and $\xi_2$ with $\sigma_1$ and $\sigma_2$, respectively. 

\textit{Step~3: $F_0^\sharp(u_0;\Omega_1)>0$.} Let $\alpha<1$ and $\eps_j^\alpha=\frac{1}{j}$ for all $j\in \N$.
Suppose by contradiction that $F_0^\sharp(u_0;\Omega_1)=0$. Then, for every $\delta>0$ there 
exists a sequence $(u_j^\delta)_j\subset L^p(\Omega_1;\R^m)$ 
with $u_j^\delta \weakly u_0$ in $L^p(\Omega_1;\R^m)$ and 
\begin{align*}
\liminf_{j\to \infty} F_{\eps_j}^\sharp(u_j^\delta; \Omega_1)<\delta.
\end{align*}
Since $f$ is coercive (compare~\eqref{coercivity_g}), $(u_j^\delta)_j$ 
is equibounded in $L^p(\Omega_1;\R^m)$, so that by a diagonalization argument one finds a sequence $(u_j)$ such that $u_j\weakly u_0$ in $L^p(\Omega_1;\R^m)$ and \begin{align*}\lim_{j\to \infty}F_{\eps_j}^\sharp(u_j;\Omega_1)=0.\end{align*}
In particular, for $j$ sufficiently large, we have that
$\Acal_{\eps_j} u_j=0$ in $\T^{d-1}\times (0,1)$. In the following let $q\in (1,p)$ be fixed.
 
To derive a contradiction, we show that for every sequence 
$(v_j)\subset L^p(\Omega_1;\R^m)$  
with $v_j\weakly u_0$ in $L^p(\Omega_1; \R^m)$ and 
\begin{align}\label{IA}
\lim_{j\to \infty} \int_{\Omega_1} f\Bigl(jx', v_j(x)\Bigr)\dd{x}=0,
\end{align}
there is a subsequence of $(v_j)$ (not relabeled) satisfying
\begin{align}\label{toprove}
\norm{\Acal_{\eps_j}v_j}_{W^{-1,q}(\T^{d-1}\times (0,1);\R^l)}
\geq c_q>0
\end{align}
for all $j\in \N$ with a constant $c_q>0$.

\textit{Step~3a: Construction of fast-oscillating functions.} 
We define for $j\in\N$,
\begin{align*}
v_j^\sharp(x)=\frac{1}{j^{d-1}}\sum_{i\in \N_0^{d-1}, \,\abs{i}<j} v_j\Bigl(\frac{i}{j}+x',x_d\Bigr)\qquad x\in \Omega_1=Q^d,
\end{align*}
where $\abs{i}:=\sum_{l=1}^{d-1} i_l$ for $i\in \N_0^{d-1}$.
By definition, $v_j^\sharp$ is $\frac{1}{j}Q^{d-1}$-periodic in the $x'$-variable, and we identify $v_j^\sharp$ with its periodic extension to $\R^d$.
In view of $v_j\weakly u_0$ in $L^p(\Omega_1;\R^m)$ with $u_0$ constant in $x'$, one finds that
(upon selection of a not relabeled subsequence) also
\begin{align}\label{convergence_vknatural}
v_j^\sharp \weakly u_0\qquad \text{in $L^p(\Omega_1;\R^m)$.}
\end{align}
To see this, we observe that $(v_j^\sharp)$ is uniformly bounded in $L^p(\Omega_1;\R^m)$ and we derive for an arbitrary cube $Q\subset\R^d$ that
\begin{align*}
\dashint_Q v_j^\sharp -u_0 \dd{x}&=\frac{1}{j^{d-1}}\sum_{i\in \N_0^{d-1}, \,\abs{i}<j}\dashint_Q  (v_j-u_0)\Bigl(\frac{i}{j}+x', x_d\Bigr)\dd{x}
\\ &\leq \dashint_Q (v_j-u_0)\Bigl(\frac{\hat{i}_j}{j}+x', x_d\Bigr)\dd{x}  
=\frac{1}{\abs{Q}} \int_{\R^d} \chi_{Q+\frac{1}{j}(\hat{i}_j, 0)}(y) (v_j-u_0)(y)\dd{y},
\end{align*}
where $\hat{i}_j=\max\{ \dashint_Q  (v_j-u_0)\bigl(\frac{i}{j}+x', x_d\bigr)\dd{x} : i\in \N_0^{d-1}, \abs{i} < j\}$. Since $|\hat{i}_j|/j<1$ for all $j\in \N$ we find that up to a subsequence $\hat{i}_j/j\to \hat{a}$ in $\R^{d-1}$ with $\hat{a}\in \R^{d-1}$. Consequently,
$\dashint_Q v_j^\sharp -u_0 \dd{x}\to 0$ as $j\to \infty$, in view of $v_j\weakly u_0$ in $L^p_{loc}(\R^d;\R^m)$ and
$\chi_{Q+\frac{1}{j}(\hat{i}_j,0)}\to \chi_{Q+(\hat{a}, 0)}$ in $L^{p'}(\R^d;\R^m)$.

Moreover, exploiting that $v_j^\sharp$ is a convex combination of translations of $u_j$ 
together with the convexity of $f$ in the second variable and the periodicity of $f$ in the first variable entails
\begin{align}\label{convergence_vnatural}
\int_{\Omega_1} f\bigl(jx', v_j^\sharp(x)\bigr)\dd{x} 
&\leq \frac{1}{j^{d-1}}\sum_{i\in \N_0^{d-1}, \,\abs{i}<j}\int_{\Omega_1} f\Bigl(jx', v_j\Bigl(\frac{i}{j}+x', x_d\Bigr)\Bigr)\dd{x}\nonumber \\ &= 
\int_{\Omega_1} f\bigl(jx', v_j(x)\bigr)\dd{x} \to 0
\end{align}
as $j\to \infty$. 
Here we used~\eqref{IA}.
Besides, from the convexity of the norm and the periodicity of $v_j$ in the first variable we conclude that for all $j\in \N$,
\begin{align}\label{Aeps}
\norm{\Acal_{\eps_j}v_j^\sharp}_{W^{-1,p}(\T^{d-1}\times (0,1);\R^l)} 
&\leq \frac{1}{j^{d-1}} \sum_{j\in \N_0^{d-1},\,\abs{i}<j} \normB{\Acal_{\eps_j} v_j\Bigl(\frac{i}{j}+\frarg, \frarg\Bigr)}_{W^{-1,p}(\T^{d-1}\times(0,1);\R^l)}\\ &= 
\norm{\Acal_{\eps_j} v_j}_{W^{-1,p}(\T^{d-1}\times(0,1);\R^l)}.\nonumber
\end{align}

\textit{Step~3b: Rescaled functions.}
We define for $x\in \Omega_1$,
\begin{align*}
\tilde{v}_j(x)=v_j^\sharp\Bigl(\frac{x'}{j},x_d\Bigr).
\end{align*}
Then, $\tilde{v}_j$ is $Q^{d-1}$-periodic for every $j\in \N$
and $(\tilde{v}_j)$ is a bounded sequence in $L^p(\Omega_1;\R^m)$, which (up to a not relabeled subsequence) generates a 
Young measure $(\tilde{\nu}_x)_{x\in \Omega_1}$.

In particular, 
\begin{align}\label{weak_convergence_vkhat}
\tilde{v}_j\weakly \int_{\R^m} a\dd{\tilde{\nu}_x}(a)=:\tilde{v}_0  \qquad \text{in $L^p(\Omega_1;\R^m)$.}
\end{align}
Accounting for \eqref{weak_convergence_vkhat},~\eqref{convergence_vknatural}, and the definition of $\tilde{v}_j$ yields that
\begin{align}\label{weaklimit_v0hat}
 \int_{Q^{d-1}}\tilde{v}_0 \dd{x'} =u_0.
\end{align}
The fundamental theorem on Young measures in conjunction with \eqref{convergence_vnatural} and a change of variables implies
\begin{align}\label{YM0}
&\int_{\Omega_1} \int_{\R^m} f(x', a) \dd{\tilde{\nu}_x}(a)\dd{x} \leq \liminf_{j\to \infty} \int_{\Omega_1}f(x', \tilde{v}_j (x))\dd{x}
= \liminf_{j\to \infty} \int_{\Omega_{1}} f\Bigl(x', v_j^\sharp\Bigl(\frac{x'}{j},x_d\Bigr)\Bigr)\dd{x}\nonumber\\
&\qquad = \liminf_{j\to \infty} j^{d-1}\int_{\frac{1}{j}Q^{d-1}\times (0,1)} f(j x', v_j^\sharp(x))\dd{x} =\lim_{j\to \infty}\int_{\Omega_1} f(j x', v_j^\sharp(x))\dd{x}= 0.
\end{align}
Here again, we exploited the $Q^{d-1}$-periodicity of $f$ in the first variable, as well as the $\frac{1}{j}Q^{d-1}$-periodicity of $v_j^\sharp$.

Moreover, we can prove that
\begin{align}\label{Aeps2}
\normb{\Acal_{\eps_j} v^\sharp_j}_{W^{-1,q}(\T^{d-1}\times (0,1);\R^l)}\geq
\normb{\Acal_{\eps^{1-\alpha}_j} \tilde{v}_j}_{W^{-1,q}(\T^{d-1}\times (0,1);\R^l)}
\end{align}
for all $j\in \N$.
This follows from the a suitable parameter transformation and the definition of the $W^{-1,q}$-norm.

Precisely, let $W_0^{1,q'}(\T^{d-1}\times (0,1);\R^m)$ denote the closure of
the set of $C^{\infty}(\T^d;\R^m)$-functions with zero boundary conditions on $Q^{d-1}\times \{0\}$ and $Q^{d-1}\times \{1\}$ regarding the norm 
$\norm{\frarg}_{W_0^{1,q'}(Q^d;\R^m)}=\norm{\nabla \frarg}_{L^{q'}(Q^d;\R^{m\times d})}$. With this definition, for $j\in \N$ let 
\begin{align*}
B_{1/j} & = \Bigl\{\varphi\in W^{1,q'}_0(\T^{d-1}\times (0,1);\R^m): \norm{\nabla \varphi}_{L^{q'}(Q^d;\R^{m\times d})}\leq 1/j\Bigr\}
\end{align*}
stand closed ball of radius $1/j$ in $W^{1,q'}_0(\T^{d-1}\times (0,1);\R^m)$,
and define
\begin{align*}
B_j^\sharp &=  \Bigl\{\varphi\in B_1: \text{$\varphi$ is $\frac{1}{j}Q^{d-1}$-periodic}\Bigr\} \subset B_1,
\end{align*}
where $B_1$ is the closed unit ball.
Straightforward computation based on the change of variables $y=(jx', x_d)$ yields that for all $j\in \N$,
\begin{align*}
j\tilde{\varphi}_j\in B_1 \qquad \Longrightarrow \qquad \tilde{\varphi}_j\in B_{1/j}\qquad  \Longrightarrow \qquad \varphi_j^\sharp \in B_j^\sharp,
\end{align*}
where $\tilde{\varphi}_j(y)=\varphi_j^\sharp(x)$ and thus, $\partial^x_i\varphi_j^\sharp(x)=j\partial_i^y\tilde{\varphi}_j(y)$ for $i=1, \ldots, d-1$.
Then, in view of
 \begin{align*}
&(\Acal_{\eps_j}^x)^T \varphi_j^\sharp(x) = (\Acal_{\eps_j}^x)^T\tilde{\varphi}_j(jx',x_d)\\ 
&\qquad = j \sum_{k=1}^{d-1} (A^{(k)})^T\partial_j^y \tilde{\varphi}_j(y) + \frac{1}{\eps_j}(A^{(d)})^T\partial_d^y\tilde{\varphi}_j(y)
=j\Acal_{\eps_j^{1-\alpha}}^y \tilde{\varphi}_j(y)\end{align*}
(recalling that $\eps_j^\alpha=1/j$), one obtains that
\begin{align*}
&\norm{\Acal_{\eps_j}v_j^\sharp}_{W^{-1,q}(\T^{d-1}\times (0,1);\R^l)}  =
\sup_{\varphi\in B_1}  \absB{\int_{Q^d} v_j^\sharp(x) \cdot \Acal_{\eps_j}^T\varphi(x)\dd{x}}\nonumber\\
& \qquad \geq \sup_{\varphi_j^\sharp \in B_j^\sharp}  \absB{\int_{Q^d} v_j^\sharp(x) \cdot \Acal_{\eps_j}^T\varphi_j^\sharp(x)\dd{x}}
\geq \sup_{j\tilde{\varphi}_j\in B_1} j \absB{\int_{Q^d} \tilde{v}_j(y) \cdot \Acal_{\eps_j^{1-\alpha}}^T\tilde{\varphi}_j(y)\dd{y}} \\
&\qquad = \sup_{\varphi \in B_1} \absB{\int_{Q^d} \tilde{v}_j(y) \cdot \Acal_{\eps_j^{1-\alpha}}^T\varphi(y)\dd{y}} = \norm{\Acal_{\eps_j^{1-\alpha}}\tilde{v}_j}_{W^{-1,q}(\T^{d-1}\times (0,1);\R^l)}\nonumber
\end{align*} for $j\in \N$. 

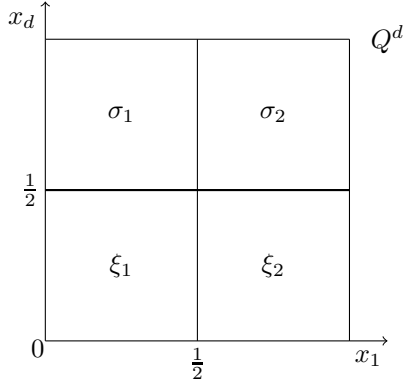
\begin{figure}[t]
\begin{tikzpicture}
\draw (0,0) -- (4,0);
\draw (0,0) --(0,-4);
\draw (4,0)--(4,-4);
\draw (0,-4)--(4,-4);
\draw[thick] (0,-2)--(4,-2);
\draw (2, 0)--(2, -4);
\draw[->] (4,-4) --node[below]{$x_1$} (4.5, -4);
\draw[->] (0,0) --node[left]{$x_d$} (0, 0.5);
\node at (1, -1) {$\sigma_1$};
\node at (3,-1) {$\sigma_2$};
\node at (1, -3) {$\xi_1$};
\node at (3,-3) {$\xi_2$};
\node at (-0.1,-4.1) {$0$};
\node at (-0.2, -2) {$\frac{1}{2}$};
\node at (2, -4.3) {$\frac{1}{2}$};
\node at (4.5,0) {$Q^d$};
\end{tikzpicture}
\caption{\label{fig:vtilde0}Visualization of $\tilde{v}_0$.}
\end{figure}

\textit{Step~3c: Characterization of the Young measure $\tilde{\nu}$.} 
In this step, we prove that for almost all $x\in \Omega_1$,
\begin{align}\label{YM_representation}
\tilde{\nu}_x=\delta_{\tilde{v}_0(x)}.
\end{align}
In doing so, we will derive the representation 
\begin{align}\label{def:vtilde0}
\tilde{v}_0(x)= \left\{\begin{array}{cl}
\sigma_1 & \text{if }  x_1\leq1/2,\, x_d>1/2,\\
\sigma_2 & \text{if }  x_1>1/2,\, x_d>1/2,\\
\xi_1 & \text{if } x_1\leq1/2,\, x_d\leq 1/2,\\
\xi_2 & \text{if }  x_1>1/2,\, x_d\leq 1/2,
\end{array}\right. \qquad x\in \Omega_1,
\end{align} 
meaning that $\tilde{v}_0$ coincides in $\Omega_1=Q^d$ with a piecewise 
constant function, see Fig~\ref{fig:vtilde0}. In the sequel, we will focus on the case $x_d> 1/2$, the arguments for $x_d<1/2$ are analogous.

From \eqref{YM0} we know that $\tilde{\nu}_x$ is supported in $\{a\in \R^m: f(x', a) =g^{\rm c}(x',a)=0\}$ 
for $x\in \Omega_1$. By the definition of
$g$, this means that 
\begin{align*}
\supp \tilde{\nu}_x\subset \begin{cases}\{\lambda \xi_1 + (1-\lambda)\sigma_1: \lambda\in [0,1]\} & \text{if $x_1\leq 1/2$,}\\
 \{\lambda \xi_2 + (1-\lambda)\sigma_2 : \lambda\in [0,1]\} & \text{if $x_1>1/2$.}
 \end{cases}
\end{align*} 
For almost every $x\in \Omega_1$ there exists a probability measure 
on $[0,1]$ named $\mu_{x}^>$ or $\mu_{x}^\leq $ such that
\begin{align*}
\tilde{\nu}_x(E)=\mu_{x}^{\leq}(\{\lambda\in [0,1]:\lambda\xi_1+(1-\lambda)\sigma_1\in E\}) \qquad \text{if $x_1\leq 1/2,$}\\
\tilde{\nu}_x(E)=\mu_{x}^{>}(\{\lambda\in [0,1]:\lambda\xi_2+(1-\lambda)\sigma_2\in E\}) \qquad \text{if $x_1>1/2,$}
\end{align*}for all $E\subset \R^m$.

Recalling that $\omega=Q^{d-1}$, we compute
\begin{align}\label{YM1}
&\int_{\omega}\int_{\R^m} a\dd{\nu_x}(a)\dd{x'} = \bar{\sigma}
+(\xi_1-\sigma_1) \int_{0}^{1/2} \int_0^1\lambda\dd{\mu_{x}^\leq}(\lambda)\dd{x_1}\nonumber\\
&\qquad\qquad\qquad\hspace{3cm}+ (\xi_2-\sigma_2) \int_{1/2}^1 \int_0^1\lambda\dd{\mu_{x}^>}(\lambda)\dd{x_1}.
\end{align}
On the other hand, in view of \eqref{weaklimit_v0hat}, we have for $x\in \Omega_1$ that
\begin{align}\label{YM2}
\int_{\omega} \int_{\R^m} a \dd{\nu_{x}}(a)\dd{x'} &= \int_{Q^{d-1}} \tilde{v}_0(x', x_d)\dd{x'} = u_0(x)=
\begin{cases} \bar{\sigma}&\text{if $x_d>1/2$,} \\ \bar{\xi} &\text{if $x_d\leq 1/2$.}\end{cases}
\end{align}
For the moment, consider a fixed $x_d\in (1/2,1)$, so that we infer from \eqref{YM1}, \eqref{YM2}, and the
linear independence of  $(\xi_2-\sigma_2)$ and $(\xi_1-\sigma_1)$ by~\eqref{lem:counterexample5} in Lemma~\ref{lem:general_counterexample} that
\begin{align*}
\int_{0}^{1/2} \int_0^1\lambda\dd{\mu_{x}^\leq}(\lambda)\dd{x_1}=0\qquad \text{and}\qquad
\int_{1/2}^{1} \int_0^1\lambda \dd{\mu_{x}^>}(\lambda)\dd{x_1}=0.
\end{align*}
This shows that for almost all $x\in \Omega_1$ with $x_d>1/2$,
\begin{align*}
\tilde{\nu}_x=\begin{cases}\delta_{\sigma_1} &\text{if $x_1\leq 1/2$,} \\
\delta_{\sigma_2}& \text{if $x_1>1/2$.}\end{cases}
\end{align*}

\textit{Step~3d: Conclusion.}
As a consequence of \eqref{YM_representation}, $\tilde{v}_j\to \tilde{v}_0$ in measure in $\Omega_1$, and therefore 
\begin{align*}
\tilde{v}_j\to \tilde{v}_0\qquad \text{in $L^q(\Omega_1;\R^m)$ for $j\to \infty$.}
\end{align*}

{\it Case A: $\alpha=1$.} In view of~\eqref{Aeps},~\eqref{Aeps2}, and the lower semicontinuity of the norm, we obtain that
\begin{align*}%\label{caseA}
\liminf_{j\to \infty}\norm{\Acal_{\eps_j}v_j}_{W^{-1, q}(\T^{d-1}\times (0,1);\R^l)}& \geq \liminf_{j\to \infty}\norm{\Acal \tilde{v}_j}_{W^{-1,q}(\T^{d-1}\times (0,1);\R^l)}\nonumber \\ &\geq \norm{\Acal \tilde{v}_0}_{W^{-1, q}(\T^{d-1}\times (0,1);\R^l)} >c_q>0.
\end{align*}
For the last step, one needs to exploit the fact that, by construction, $\tilde{v}_0$ is not $\Acal$-free as a consequence of~\eqref{lem:counterexample4}.

{\it Case B: $\alpha<1$.}
By the continuity of first-order derivatives as linear operators from $L^q$ to $W^{-1,q}$, we get
\begin{align}\label{A0neq0}
\norm{\Acal_0 \tilde{v}_j}_{W^{-1,q}(\T^{d-1}\times (0,1);\R^l)} \to 
\norm{\Acal_0 \tilde{v}_0}_{W^{-1,q}(\T^{d-1}\times (0,1);\R^l)}\neq 0.
\end{align}
Notice that $\Acal_0 \tilde{u}_0 \neq 0$ due to Lemma~\ref{lem:Afree_jumps} and $\xi_1-\sigma_1\notin \ker \Abb_0(e_d)$
or $\xi_2-\sigma_2\notin \ker \Abb_0(e_d)$ according to~\eqref{lem:counterexample4} in Lemma~\ref{lem:general_counterexample}. 

To conclude the proof of Step~3, we show that there is a subsequence (not relabeled) and a constant $c_q>0$ such that 
\begin{align*}
\norm{\Acal_{\eps_j^{1-\alpha}}\tilde{v}_j}_{W^{-1,q}(\T^{d-1}\times (0,1);\R^l)} \geq c_q
\end{align*}
for all $j\in \N$, which in view of~\eqref{Aeps} and~\eqref{Aeps2} finishes the proof of~\eqref{toprove}.
In fact, assuming that $\lim_{j\to \infty}\norm{\Acal_{\eps_j^{1-\alpha}}\tilde{v}_j}_{W^{-1,q}(\T^{d-1}\times (0,1);\R^l)} =0$ implies
\begin{align*}
\lim_{j\to \infty} \norm{\Acal_0\tilde{v}_j}_{W^{-1,q}(\T^{d-1}\times (0,1);\R^m)}=0
\end{align*} 
by Lemma~\ref{lem:convergence_AepsA0} applied with $\beta=1-\alpha>0$, which is in contradiction to~\eqref{A0neq0}.
\end{proof}

The next two technical lemmata were used in the proof of Proposition~\ref{prop:nonlocal_alpha<1}.
\begin{lemma}[$\Acal$-free jumps]\label{lem:Afree_jumps}
Let $\Acal$ be a first-order partial differential operator with constant coefficients as in Section~\ref{sec:formulation_Afree}.
Further, let $n\in \R^d$ and $\xi, \sigma\in \R^m$, and consider the jump function 
$v:\R^d\to\R^m$ given by
\begin{align*}
v(y)= \begin{cases} 
\xi & \text{ if $y\cdot n\leq 0$,}\\ 
\sigma & \text{ if $y\cdot n > 0$,}
\end{cases} \qquad y\in \R^d.
\end{align*}
Then, $\Acal v=0$ in $\R^d$ if and only if $\xi-\sigma \in \ker \Abb(n)$.
\end{lemma}

\begin{proof}
From \cite[Remark~2.5]{KR14}, which is essentially a generalization of the 
classical Gau{\ss}-Green theorem, one infers for all  $\varphi\in C_0^\infty(\R^d;\R^l)$ that
\begin{align*}
\int_{\R^d}  v\cdot \Acal^T\varphi \dd{y}&=\int_{\{n\cdot y >0\}} v\cdot \Acal^T\varphi \dd{y} +
 \int_{\{n\cdot y <0\}}  v\cdot\Acal^T \varphi \dd{y} \\ & = \int_{\{n\cdot y=0\}} \Abb(n)\sigma\cdot \varphi \dd{\Hcal^{d-1}} 
 + \int_{\{n\cdot y=0\}} \Abb(-n) \xi \cdot \varphi\dd{\Hcal^{d-1}}  \\ 
 &= \int_{\{n\cdot y=0\}} \Abb(n) (\xi-\sigma)\cdot \varphi \dd{\Hcal^{d-1}},
\end{align*}
where $\Hcal^{d-1}$ is the $(d-1)$-dimensional Hausdorff measure.
\end{proof}

\begin{lemma}\label{lem:convergence_AepsA0}
Let $\eps_j\todown 0$, $\beta>0$ and $q>1$. If $(v_j)\subset L^q(Q^d;\R^m)$ is a bounded sequence and
 $\Acal_{\eps_j^\beta}v_j\to 0$ in $W^{-1,q}(\T^{d-1}\times (0,1);\R^l)$,
 then 
 \begin{align*}
 \Acal_0 v_j\to 0 \qquad\text{in $W^{-1,q}(\T^{d-1}\times (0,1);\R^l)$.}
 \end{align*}
\end{lemma}

\begin{proof}
As in~\eqref{splittingA} we split the operators $\Acal_{\eps_j^\beta}$ and $\Acal_0$ into
\begin{align*}
\Acal_{\eps_j^\beta}=\left(\frac{ (\Acal_{\eps_j^\beta})_+}{(\Acal_{\eps_j^\beta})_-}\right)\qquad \text{and}\qquad \Acal_0=\left(\frac{ (\Acal_0)_+}{(\Acal_0)_-}\right).
\end{align*}
Then, 
\begin{align*}
(\Acal_0)_+v_j=A_+^{(d)}\partial_d v_j = \eps_j^\beta (\Acal_{\eps_j^\beta})_+ v_j -\eps_j^\beta \Acal_+'v_j\to 0\qquad \text{in $W^{-1,q}(\T^{d-1}\times (0,1);\R^r)$}
\end{align*}  
as $j\to \infty$. In the first term, the assumption was used, while the convergence of the second term follows from $\beta>0$ and the fact that
$\norm{\Acal'_+ v_j}_{W^{-1,q}(\T^{d-1}\times (0,1);\R^r)} \leq c\norm{v_j}_{L^q(Q^d;\R^m)}$ is uniformly bounded.
On the other hand, for $(\Acal_0)_-$ one finds that
\begin{align*}
(\Acal_0)_- v_j=\Acal'_- v_j =( \Acal_{\eps_j^\beta})_-v_j \to 0 \qquad \text{in $W^{-1,q}(\T^{d-1}\times (0,1);\R^{l-r})$}
\end{align*}
for $j\to \infty$.
\end{proof}

%%%%%%%%%%%%%%%%%%%%%%%%%%%%%%%%%%%%%%%% Successive limits II: Homogenization of the thin-film limit %%%%%%%%%%%%%%%%%%%%%%%%%%%%%%%%%%%%%%%%%%%%

\section{Homogenization of the thin-film limit}\label{sec:epsdelta}
This section is concerned with the double limit
\begin{align}\label{doublelimit_epsdelta}
\Gamma\text{-}\lim_{\delta\to 0}[\Gamma\text{-}\lim_{\eps\to 0} F_{\eps, \delta}] =\Gamma\text{-}\lim_{\delta\to 0}F_\delta.
\end{align} 
As in the corresponding simultaneous case in Section~\ref{sec:alpha<1}, the goal is to find a suitable integrand $f$ such that the expression in~\eqref{doublelimit_epsdelta} is nonlocal.

Even though the limit functional $F_\delta$ for fixed $\delta>0$ is known to be local for convex integrands with no dependence on $z_d$, precisely,
\begin{align*}
F_\delta(u) = 
\begin{cases} 
\int_{\Omega_1} f(\frac{y'}{\delta}, u(y))\dd{y} & \text{if $\Acal_0 u=0$ in $\Omega_1$,} \\
\infty & \text{else,}
\end{cases}\qquad u\in L^p(\Omega_1;\R^m),
\end{align*}
by Theorem~\ref{theo:dimred},
it is impossible to apply the homogenization result of Theorem~\ref{theo:hom} to $F_\delta$, because this would require the differential constraint to be conveyed by a constant-rank operator, which $\Acal_0$ in general is not (see~\cite[Section~2.7]{KR14}).

Exactly the same construction as in Section~\ref{sec:alpha<1} can be used here to show the nonlocality of the successive limits in~\eqref{doublelimit_epsdelta}. The desired incompatibility is even easier to see.
In view of $\ker \Abb(n)\subset \ker \Abb_0(n)$ for $n=(n',0)\in \Scal^{d-1}$, we find that the quantity $\tilde{v}_0$ defined in~\eqref{def:vtilde0} fails to be $\Acal_0$-free in $\T^{d-1}\times(0,1)$. 
This yields the following result:

\begin{proposition}[Nonlocal character of \boldmath{$F_0^\sharp$}]
Let $\Acal$ and $n$ be as in Lemma~\ref{lem:generalized_vectorquartett}, assuming that $n\neq e_d$ is a 
standard unit vector in $\R^d$, and consider a given sequence {$\delta_j\todown 0$}. Further, let $f:\R^{d-1}\times \R^m \to \R$ be the function defined in \eqref{def:f}, 
and $\Omega_1=\omega\times(0,1)=Q^d$. For an open set $D\subset \Omega_1$ the $\Gamma$-$\liminf$ is denoted by
\begin{align*}
F_0^\sharp(u;D)&=\inf \{\liminf_{j\to \infty} \textstyle \int_D f (\frac{x'}{\delta_j}, u_j(x))\dd{x}: (u_j)\subset L^p(\Omega_1;\R^m), \\ & \qquad\qquad\qquad \Acal_0 u_j=0 \text{ in $\T^{d-1}\times (0,1)$}, u_j\weakly u \text{ in $L^p(\Omega_1; \R^m)$}\}
\end{align*}
for $u\in L^p(\Omega_1;\R^m)$, and for $x\in \Omega_1$ let
\begin{align*}
u_0(x)=\begin{cases} \bar{\sigma}& \text{if $x_d>1/2$,}\\
\bar{\xi} & \text{if $x_d\leq 1/2$},
\end{cases}
\end{align*} 
where $\bar{\xi}=\frac{1}{2}(\xi_1+\xi_2)$ and $\bar{\sigma}=\frac{1}{2}(\sigma_1+\sigma_2)$.

Then,  
\begin{align*}
F_0^\sharp(u_0; \omega \times (0,1/2)) = 0 = F_0^\sharp(u_0; \omega \times (1/2, 1)).
\end{align*}
If $\delta_j=\frac{1}{j}$ for all $j\in \N$, then
$F_0^\sharp(u_0; \Omega_1)>0$.
\end{proposition}

\begin{remark}[Lower bound]
Recalling~Remark~\ref{rem:upperbound_twoscale} on weak two-scale limits of $\Acal_0$-free fields, we know that~\eqref{doublelimit_epsdelta} is bounded from below by the local expression $\int_{\Omega_1} \fhomzero(u)\dd{x}$, if $u\in \Ucal_{\Acal_0}$.
\end{remark}

 %%%%%%%%%%%%%%%%%%%% ACKNOWLEDGEMENTS %%%%%%%%%%%%%%%%%%%%%%%%%%%%
\section*{Acknowledgements}
The research of SK was supported by the project CZ01-DE03/2013-2014/DAAD-56269992 (PPP program).
CK acknowledges the support of the Funda\c{c}a\~{o} para a Ci\^{e}ncia e a Tecnologia (Portuguese Foundation for Science and Technology) through the ICTI CMU-Portugal Program in Applied Mathematics and UTA-CMU/MAT/0005/2009, as well as of  ERC-2010-AdG no.~267802 ``Analysis of Multiscale Systems Driven by Functionals''. 
%%%%%%%%%%%%%%%%%%%% BIBLIOGRAPHY %%%%%%%%%%%%%%%%%%%%%%%%%%%%%%%%%%%%

\bibliographystyle{acm}
\bibliography{Afree}

\end{document}